\documentclass[reqno]{amsart}
\usepackage[T1, T2A]{fontenc}
\usepackage[utf8]{inputenc} 
\usepackage{graphicx}
\usepackage{amssymb,latexsym}
\usepackage{subfigure, url}
\usepackage{tikz}
\usepackage{color}

\numberwithin{equation}{section}

\def\eps{{\varepsilon}}

\def\C{{\mathbb C}}
\def\D{{\mathbb D}}

\def\H{{\mathbb H}}
\def\N{{\mathbb N}}

\def\R{{\mathbb R}}

\def\Z{{\mathbb Z}}
\def\CP{{\mathbb C\mathbb P}}

\def\ov{\overline}

\def\zbar{\overline{z}}
\def\epsbar{\overline{\varepsilon}}
\def\sgn{{\rm sgn}}
\newcommand\Tt{{T_{\scriptscriptstyle{1\over 2}}}}
\newcommand\STt{\Sigma\circ\Tt}
\def\Re{\rm Re}
\def\Im{\rm Im}

\usepackage{cleveref}
\theoremstyle{plain}
\newtheorem{lemma}{Lemma}[section]
\newtheorem{proposition}[lemma]{Proposition}

\theoremstyle{definition}

\newtheorem{remark}[lemma]{Remark}
\theoremstyle{plain}

\newtheorem{theorem}[lemma]{Theorem}
\newtheorem{corollary}[lemma]{Corollary}

\theoremstyle{definition}
\newtheorem{definition}[lemma]{Definition}

\theoremstyle{remark}
\newtheorem{notation}[lemma]{Notation}

\title[Unfoldings of antiholomorphic parabolic point]{Generic unfolding of an antiholomorphic parabolic point of codimension $k$\footnote{The author is supported by NSERC in Canada.}
}

\author[C. Rousseau]{Christiane Rousseau}
\address{D\'epartement de math\'ematiques et de statistique, Universit\'e de Montr\'eal, C.P. 6128, Succursale Centre-ville, Montr\'eal (Qc), H3C 3J7, Canada.}
\email{christiane.rousseau@umontreal.ca}

\usepackage{xcolor}
\usepackage{xspace}


\begin{document}
\date{\today}
\maketitle

\begin{abstract}
We classify generic unfoldings of germs of antiholomorphic
  diffeomorphisms with a parabolic point of codimension~$k$ (i.e.~a fixed point of multiplicity $k+1$) under conjugacy. Such generic unfoldings depend real analytically on $k$ real parameters. A preparation of the unfolding allows to identify  real analytic \emph{canonical parameters}, which are preserved by any conjugacy between two prepared generic unfoldings. A modulus of analytic classification is defined, which is an unfolding of the modulus assigned to the antiholomorphic parabolic point. Since the second  iterate of such a germ is a real unfolding of a holomorphic parabolic point, the modulus is a special form of an unfolding of the \'Ecalle-Voronin modulus  of the second  iterate of the antiholomorphic parabolic germ. We also solve the problem of the existence of an antiholomorphic square root to a germ of generic analytic unfolding of a holomorphic parabolic germ.   \end{abstract}
\smallskip
\noindent

\section{Introduction} 

Antiholomorphic dynamics is developing in parallel with holomorphic dynamics. The development of  holomorphic dynamics has taken off from the fine study of the structure of the Mandelbrot set for quadratic polynomials by Douady and Hubbard (\cite{DH1} and \cite{DH2}). The Mandelbrot set was further generalized to multibrot sets for polynomials of higher degree. But in the cubic case the multibrot is not locally connected. 
To further investigate the cubic case, Milnor studied real cubic polynomials in 1992 (see \cite{M}). There, a prototype for the behavior in the bitransitive case was the tricorn, which is the equivalent of the Mandelbrot set for the antiholomorphic map $z\mapsto \ov{z}^2+c$. The generalization of the tricorn was the multicorn which appears for $z\mapsto \ov{z}^d+c$. This made the link between holomorphic and antiholomorphic dynamics and led to an increasing interest in the latter. 

Considering holomorphic dynamics, for instance iterations of quadratic polynomials, the interesting behavior occurs close to the boundary of the Mandelbrot set. There, periodic points with rational multipliers (also called resonant periodic points) are dense and organize the global dynamics. The local study of these periodic points sheds some light on how this dynamics is organized. 

In parallel, a whole chapter of mathematics developed around the classification problem for singularities in analytic dynamics. \'Ecalle (\cite{E}) and Voronin (\cite{V}) classified resonant fixed points of germs of $1$-dimensional analytic diffeomorphisms 
\begin{equation} f(z) = \exp\left(\frac{2\pi i p}{q}\right)z+ z^{kq+1} + O(z^{kq+2})\label{par}\end{equation}
up to conjugacy (local changes of coordinates) and derived moduli spaces for these. The moduli are constructed as follows. While a simple formal normal form exists, the formal normalizing change of coordinate generically diverges. But there exists almost unique normalizing changes of coordinates on sectors covering a punctured neighborhood of the fixed point. The modulus is given by the mismatch between these almost unique normalizing changes of coordinates. The moduli spaces are huge, namely functional spaces, thus highlighting the richness of the different geometric behaviors of these singularities. Explaining this richness came from two directions. To highlight this, let us focus on  the simplest case of a double singular point, called a codimension $1$ parabolic point ($p=q=k=1$ in \eqref{par}). The normal form in this case is the time-one map of the flow of a vector field $\frac{z^2}{1+bz}\;\frac{\partial}{\partial z} $. Since a double fixed point can be seen as the merging of two simple fixed points, it is natural to unfold the germ of analytic diffeomorphism in a family splitting the double fixed point into two simple fixed points. Two independent attempts to understand the dynamics developed in parallel. On the one hand, there were studies in the parameter directions in which the simple fixed points where linearizable (see for instance \cite{Ma} and \cite{Gl}). In the neighborhood of each fixed point the diffeomorphism is analytically conjugate to the normal form given by the  time-one map of the flow of a vector field $\frac{z^2-\eps}{1+b(\eps)z}\;\frac{\partial}{\partial z} $. But, generically the two normalizations do not match. The mismatch is a modulus of the unfolding for these parameter values and the limit of this mismatch when the fixed points merge together is the \'Ecalle-Voronin modulus. This approach could not work in the parameter directions where either at least one simple fixed point is not normalizable or the domains of normalizations have void intersections. A way through came from a visionary idea of Douady, namely to  normalize the system in some domains that contains sectors at the two fixed points and whose union cover a punctured neighborhood of the two fixed points. If the domains are appropriately chosen, then the normalizations are almost unique, thus allowing to unfold the moduli.  This approach was  first proposed in the thesis of Lavaurs (\cite{L}) and normalizing coordinates were constructed by Shishikura \cite{S}. The method could be generalized to cover all directions in parameter space and led to constructions of moduli for germs of unfoldings of parabolic points (\cite{MRR} for the generic case and \cite{Ri} for the general case). The generalization involves taking domains spiraling when approaching the fixed points. Furthermore, the moduli space was identified in \cite{CR}.

Generalizations to parabolic fixed points of multiplicity $k+1$ (i.e. codimension $k$) were made possible again through the visionary ideas of Douady, who sensed that the structure of domains on which to perform the normalizations was linked to the dynamics of polynomial vector fields $P(z) \frac{\partial}{\partial z}$ on $\C$. In that case a full generic unfolding involves $k$ independent parameters. The first step performed by Oudkerk (\cite{O}) covered some directions in parameter space. A few years later, the systematic study of the generic polynomial vector fields was finalized in \cite{DES05}. Using these results, the methods of \cite{MRR} can be generalized to cover the full parameter space, Again, almost unique normalizations exist on domains which have spiraling sectors attached to two fixed points. These can be used to define a modulus of analytic classification for generic germs of unfoldings of parabolic fixed points of codimension $k$ (\cite{R15}). (Note that \cite{Ri} treats the case of $1$-parameter unfoldings.) Identifying the moduli space is still open for $k>1$. 

A similar program can be carried for multiple fixed points (also called parabolic points) of germs of antiholomorphic diffeomorphisms \begin{equation} f(\zbar) = \zbar\pm \frac12 \zbar^{kq+1} + O(\zbar^{kq+2})\label{antipar}\end{equation} and their unfoldings. The analytic classification of such germs was done in \cite{GR21}.
The similarities with the holomorphic case come from the fact that the second iterate of an antiholomorphic map is holomorphic, and hence results on holomorphic parabolic points are relevant. The differences are at the parameter level. The holomorphic or antiholomorphic dependence of  
 an antihomorphic diffeomorphism on parameters is not preserved by iteration. This comes from the fact that the condition for a multiple fixed point to have multiplicity $k+1$ has real codimension $k$ and a generic unfolding depends  real-analytically of $k$ real parameters. The classification problem of codimension $1$ unfoldings (parabolic points of multiplicity $2$) has been completely studied in \cite{GR22}, including identifying the moduli space.
 
 In this paper we consider the higher codimension $k$ case. Usually, a conjugacy of parametrized families of dynamical systems involves a change of parameter, which governs which member of the first family is conjugate to which member of the second family. In a generic holomorphic unfolding of a parabolic germ, there is a choice of a canonical multi-parameter $\eps=(\eps_0, \dots, \eps_{k-1})$, which is unique up to the action of the rotation group of order $k$. 
 A modulus of analytic classification for such a generic unfolding $g_\eps$ is given by a measure of how much $g_\eps$ differs from its formal normal form given by the time one map $v_\eps^1$ of a vector field \begin{equation}v_\eps=\frac{z^{k+1} + \eps_{k-1}z^{k-1} +\dots+ \eps_1 z + \eps_0}{1+b(\eps)z^k}\,\frac{\partial} {\partial z}.\label{intro:v}\end{equation}  The normal form is invariant under $(z,\eps_0, \dots, \eps_{k-1}) \mapsto \left(\tau z, \tau\eps_0, \dots, \tau^{-(k-2)}\eps_{k-1}\right)$, with $\tau^k=1$. And, in the particular case where $\ov{b(\eps)}=b(\epsbar)$, then for real $\eps$ there are $k$ invariant lines under the dynamics, and each choice of canonical parameter is associated to an invariant line. 
 
 In the antiholomorphic case, we consider generic unfoldings depending real-analytically on $k$ real parameters. We show that for $k$ odd, there is a unique choice of canonical parameters. For $k$ even, the only freedom is the action on parameters of $z\mapsto -z$. Hence (up to conjugating with $z\mapsto -z$ when $k$ is even) any conjugacy between two unfoldings must preserve the canonical parameters. Moreover, a change of coordinate and move to the canonical parameters \emph{prepares} the family to a form $f_\eps$ naturally compared to a formal normal form, where $\eps=(\eps_0, \dots \eps_{k-1})$ is a real-analytic multi-parameter.  This normal form is given by $\sigma\circ v_\eps^{\frac12}$, where $v_\eps$ is defined in \eqref{intro:v} and $\sigma$ is the complex conjugation, and $b(\eps)$ is always real. Note that this normal form has no rotational symmetry (except under $z\mapsto -z$ when $k$ is even). Moreover,  the real axis is the only invariant line and a symmetry axis for \eqref{intro:v}.
 
In practice, to derive a modulus it is useful to extend $\eps$ to $\C^k$ and $f_\eps$ antiholomorphically in the parameter. Then the diffeomorphism $g_\eps= f_{\epsbar}\circ f_\eps$ is a holomorphic unfolding of a holomorphic parabolic point of codimension $k$ depending holomorphically on the complex parameter $\eps\in\C^k$. A modulus of analytic classification for $g_\eps$ is given by a measure of how much $g_\eps$ differs from its formal normal form. As a result, a modulus in the antiholomorphic case is obtained from the fact that two prepared families $f_{1,\eps}$ and $f_{2,\eps}$ are analytically conjugate under a conjugacy tangent to the identity if and only if their associated \lq\lq squares\rq\rq\ defined by $g_{j,\eps}= f_{j,\epsbar}\circ f_{j,\eps}$ are holomorphically conjugate under a conjugacy tangent to the identity.

We then consider several applications. As a first one, we derive the necessary and sufficient condition for the existence of an invariant real analytic curve for real values of the parameters. Of course, this curve can be rectified to the real axis. In the second application, we consider the necessary and sufficient conditions under which a germ of generic unfolding of holomorphic parabolic germ $g_\eps$ has an \lq\lq antiholomorphic square root\rq\rq, i.e. can be decomposed as $g_\eps= f_{\epsbar}\circ f_\eps$, with $f_\eps$ antiholomorphic. These conditions are just the unfoldings of the corresponding conditions for the germ at $\eps=0$ given in \cite{GR21} and consist in some symmetry property of the modulus. As a particular case, we show that the quadratic family $g_\eps(z)= z + z^2-\eps$ has no antiholomorphic square root for small $\eps$. 

As a last application, we consider the map $\zbar^d+c$, for $c\in C$, and the associated multicorn for an integer $d\geq2$. It is known that there are exactly $d+1$ values of $c$ for which there exists a parabolic fixed point of codimension greater than $1$ (i.e. multiplicity greater than $2$). We show that these points have exact codimension 2 and that the family $\zbar^d+c$ is a generic unfolding of these points. 

\section{Preparation of the family} 

\subsection{Generalities and notations} 

\begin{notation}\label{notation} \begin{enumerate}
\item We denote by $T_a$ the translation by $a\in \C$. 
\item We denote by  $\sigma$ the complex conjugation $z\mapsto \zbar$.
\item We denote by $\D_r$ the disk of radius $r$.
\end{enumerate}
\end{notation}

\begin{definition} A map $f$ defined on a domain of $\C$ is \emph{antiholomorphic} if $\frac{\partial f}{\partial z}=0$, which is equivalent to $\sigma\circ f$ being holomorphic. 
\end{definition}

\begin{remark} Let $z_0$ be a fixed point of a antiholomorphic map $f$. Then only $|f'(z_0)|$ is an analytic invariant under analytic changes of coordinates. \end{remark}

\begin{definition} 
A multiple fixed point of finite multiplicity of a germ of holomorphic or  antiholomorphic diffeomorphism  is called \emph{parabolic}. 
The germ is said to be holomorphically parabolic or antiholomorphically parabolic.  \end{definition}

\begin{proposition}\cite{GR21} Let $z_0$ be a parabolic fixed point of a germ of antiholomorphic diffeomorphism. Then there exists a holomorphic change of coordinate in the neighborhood of $z_0$ bringing the diffeomorphism to the form $$
f_0(z) =  \begin{cases} \ov{z} +\frac12 \ov{z}^{k+1} +\left(\frac{k+1}8-\frac{b}2\right)\ov{z}^{2k+1}+ o(\ov{z}^{2k+1}), &k\:\text{odd},\\
\ov{z} \pm\frac12 \ov{z}^{k+1} + \left(\frac{k+1}8-\frac{b}2\right)\ov{z}^{2k+1}+ o(\ov{z}^{2k+1}), &k\:\text{even}, \end{cases}$$
with $b\in \R$. The integer $k>1$ is called the \emph{codimension}, and the number $b$ is the \emph{formal invariant}. The same $k$ and $b$ are the codimension and formal invariant of the holomorphic parabolic germ $g_0=f_0\circ f_0$.
\end{proposition}

\begin{remark} Note that when $k$ is even, if we have the minus sign in $f_0$, then we have the plus sign in $f_0^{-1}$. Hence we limit ourselves to the plus sign. \end{remark}

In this paper we consider germs of families of antiholomorphic diffeomorphisms depending real-analytically on $k$ real parameters and unfolding a parabolic germ of the form
\begin{equation}f_0(z) =  \ov{z} +\frac12 \ov{z}^{k+1} +\left(\frac{k+1}8-\frac{b}2\right)\ov{z}^{2k+1}+ o(\ov{z}^{2k+1}). \label{germ_0}\end{equation} 

The germs of families have the form
\begin{equation} f_\eta(z) = \ov{z} +\sum_{j=0}^{k+1} a_j(\eta)\ov{z}^j+\frac12 \ov{z}^{k+1} + o(\ov{z}^{k+1}),\label{gen_unfolding}\end{equation}
with $a_j(0)=0$ and $\eta= (\eta_0, \dots, \eta_{k-1})\in (\R^k,0)$. 

\begin{definition} The family \eqref{gen_unfolding} is \emph{generic} if the change of parameters $\eta\mapsto ({\Re}(a_0), \dots,  {\Re}(a_{k-1}))$ is invertible. \end{definition}

The second iterate $g_\eta=f_\eta\circ f_\eta$ is  an unfolding of the holomorphic parabolic germ depending on $k$ real parameters, but it will be useful to complexify the parameters. The following lemma is obvious.  

\begin{lemma}\label{lem:complexify} Let us complexify the parameters $\eta$ in $f_\eta$ in such a way that $f_\eta$ depends antiholomorphically on $\eta$ (i.e. $\frac{\partial f_\eta}{\partial \eta_j}=0$, $j=0, \dots, k-1$). Then the map $g_\eta$ defined for complex $\eta$ by
\begin{equation} g_\eta = f_{\ov{\eta}}\circ f_\eta. \label{def:g}\end{equation}
is a generic full unfolding of $g_0$ depending holomorphically on $\eta\in (\C^k,0)$. \end{lemma}
\begin{proof} Note that $g_\eta$ depends holomorphically from $\eta$. 
Moreover the $a_j$ are antiholomorphic in $\eta$, i.e. functions $a_j(\ov{\eta})$. Then $$g_\eta(z) = z +\sum_{j=0}^{k+1}\left(2{\rm Re}(a_j (\eta)) +o(\eta))\right)z^j+z^{k+1}(1+ O(\eta)) + o(z^{k+1}),$$ from which the genericity follows.\end{proof}

But for the time being, we continue with $\eta\in (\R^k,0)$.

\begin{lemma}\label{lemma:multipliers} Let $f$ be an antiholomorphic diffeomorphism, and $g=f\circ f$ be its second iterate. If $z_0$ is a fixed point of $f$ then $g'(z_0)\in\R_{\geq0}$. 
If $\{z_1,z_2\}$ is a periodic orbit of period $2$ of $f$, then $g'(z_1)=\ov{g'(z_2)}$. \end{lemma}
\begin{proof} We have $g'(z_0)= f'(z_0)\ov{f'(z_0)}$. Also  $g'(z_1) = f'(z_2)\ov{f'(z_1)}$ and $g'(z_2) = f'(z_1)\ov{f'(z_2)}$, from which the result follows. 
\end{proof}

\begin{corollary} Let $f_\eta$ be an unfolding of an antiholomorphic parabolic germ and let $g_\eta=f_{\ov{\eta}}\circ f_\eta$ be its second iterate.
Then its formal invariant $b(\eta)$ commutes with $\sigma$. \end{corollary}
\begin{proof} Let $z_0, \dots, z_k$ be the fixed points and periodic points of period $2$ of $f_\eta$ merging to the origin for $\eta=0$: these are the fixed points of $g_\eta$. It is known (see for instance \cite{R15}) that $b(\eta)= \sum_{s=0}^{k} \frac1{\log g_\eta'(z_s)}$, which is real for real $\eta$ by Lemma~\ref{lemma:multipliers}. \end{proof}

We want to classify germs of unfoldings of antiholomorphic parabolic germs under conjugacy by \emph{mix analytic} fibered changes of coordinate and parameters.

\begin{definition} A change of coordinate and parameter, $(z_1,\eta)\mapsto (z_2, \eps)=\left(H(z_1,\eta), \phi(\eta)\right)$, is \emph{mix analytic} if 
\begin{itemize}
\item it is a diffeormorphism defined on a neighborhood $\D_r\times \prod_{\ell=0}^{k-1}(-\delta_\ell,\delta_\ell)$ of $0\in \C\times \R^k$, where $\D_r$ is the disk of radius $r$; 
\item $\phi$ depends real-analytically of $\eta$; 
\item $H$ depends holomorphically on $z_1$ and real-analytically on $\eta$. \end{itemize}\end{definition}

\begin{definition} Two germs $f_{1,\eta}$ and $f_{2,\eps}$ of unfoldings of antiholomorphic parabolic germs are \emph{conjugate} if there exists a mix analytic change of coordinate and parameters $(z_1,\eta)\mapsto (z_2, \eps)=\left(H(z_1,\eta), \phi(\eta)\right)$ defined on some $R=\D_r\times \prod_{\ell=0}^{k-1}(-\delta_\ell,\delta_\ell)$ such that for all $(z_1,\eta)\in R$
$$H(f_{1,\eta}(z_1),\eta)=f_{2,\phi(\eta)}(H(z_1,\eta)).$$
\end{definition}

\subsection{Preparing the family}

\begin{theorem}\label{thm_prepared} We consider a germ of generic $k$-parameter family unfolding an antiholomorphic parabolic germ of the form \eqref{gen_unfolding}. There exists a mix analytic (fibered) change of coordinate and parameters $(z,\eta)\mapsto (Z,\eps)$ transforming \eqref{gen_unfolding} to 
$$F_\eps(Z)=\ov{Z}+ P_\eps(\ov{Z})\left(\frac12+Q_\eps(\ov{Z})+P_\eps(\ov{Z})R_\eps(\ov{Z})\right),$$
where \begin{itemize}
\item $P_\eps(\ov{Z})= \ov{Z}^{k+1} +\sum_{j=0}^{k-1} \eps_j\ov{Z}^j$ and $Q_\eps$ is a polynomial of degree at most $k$ with real analytic coefficients in $\eps$;
\item if $Z_1, \dots, Z_{k+1}$ are the fixed points and periodic points of period 2 of $F_\eps$, i.e. the fixed points of $G_\eps=F_\eps^{\circ 2}$, then
$b(\eps):=\sum_{s=1}^{k+1}  \frac1{\log G_\eps'(Z_s)}$ is real analytic with real values;
\item if $v_\eps= \frac{P_\eps(Z)}{1+b(\eps)z^k}\,\frac{\partial}{\partial z}$, then $\log F_\eps'(Z_s)=\frac12\ov{v_\eps'(Z_s)}$ for $s=1, \dots, k+1$. 
\end{itemize}
\end{theorem}
 
\begin{proof} Let us consider the  fixed points of $f_\eta$. 
Taking $z= x+iy$, this leads to the two equations
\begin{align}\begin{split} 
0 &= \sum_{j=0}^{k+1} {\Re}(a_j) \left( x^j+y^2O(|x,y|^{j-2})\right) +\frac12x^{k+1}\left(1+O(\eta)+O(x)) +y^2O(|x,y|^{k-1}\right) \\
&\qquad+\sum_{j=1}^{k-1} {\Im}(a_j) y\,O(|x,y|^{j-1})+\dots,\\
0 &=-2y+ O(\eta) +o(|x,y|),\end{split}\end{align}
where coefficients of terms with negative exponent vanish. 
The second equation can be solved by the implicit function theorem, yielding $y=h(\eta,x)= O(\eta)+o(x)$, with $h$ real analytic in $(x,\eta)$.
Replacing this in the first equation yields
\begin{align}\begin{split}0&=\sum_{j=0}^{k} \left({\Re}(a_j)+O(|a_0|,\dots,  |a_{j-1}|)+ o(\eta)\right)x^j \\
&\qquad+ \frac12(1+O(\eta))x^{k+1}+ o(x^{k+1}).\label{eq_x}\end{split}\end{align}
By the Weierstrass preparation theorem in the real analytic case, then \eqref{eq_x} is equivalent to 
$P_{0,\eta}(x)=0$, with $P_{0,\eta}$ a Weierstrass polynomial of the form:
$$P_{0,\eta}(x) = \sum_{j=0}^k 2\left({\Re}(a_j)+O(|a_0|,\dots,  |a_{j-1}|)+o(\eta)\right)x^j  +x^{k+1}.$$
We make the change of variable $z= z_1+ih(\eta,z_1)$, which sends the real axis in $z_1$-space to $y=h(x)$ in $z$-space. Let $f_{1,\eta}$ be the expression of $f_\eta$ in the new variable $z_1$. 
Then all fixed points of $f_{1,\eta}$ occur on the real line in $z_1$-space. Moreover, if $z_1=x_1+iy_1$,  the equation for the  fixed points of $f_1$ has the same form as before: $y_1=0$ and
$$0=P_{1,\eta}(x_1) = \sum_{j=0}^k 2\left({\Re}(a_j)+O(|a_0|,\dots,  |a_{j-1}|)+o(\eta)\right)x_1^j  +x_1^{k+1}.$$

The next step is to make a translation by a real number $z_2= z_1+O(|a_0|,\dots,  |a_{k-1}|)+{\Re}(a_k)+o(\eta)$ transforming $P_{1,\eta}(x_1)$ to 
$$P_{2,\eta}(x_2) = \sum_{j=0}^{k-1} \alpha_j(\eta)x_2^j  +x_2^{k+1},$$
where $\alpha_j(\eta)= 2({\Re}(a_j)+o(\eta)).$ Let $\alpha= (\alpha_0, \dots, \alpha_{k-1})$. If the family is generic, the change of parameters $\eta\mapsto\alpha$ is invertible and we could as well take $\alpha$ as new parameter. But, in practice we will keep $\eta$.

\bigskip
When considering $f_\eta$ as a 2-dimensional real diffeomorphism, the eigenvalues at a fixed point are two opposite real numbers $\pm\lambda$ and determined by a unique real number $\lambda$ (this corresponds to the fact that only the norm of $f_\eta'(\lambda)$ is intrinsic). 

\bigskip If $f_{2,\eta}$ is the expression of $f_\eta$ in the variable $z_2$ and $z_2=x_2+iy_2$, then the fixed points of $f_{2,\eta}$ are the points $x_2+ i\cdot 0$, where $x_2$ is a real solution of $P_{2,\eta}(x_2)=0$, and there exists an open set in $\eta$-space in which $P_{2,\eta}$ has $k+1$ real roots corresponding to $k+1$ fixed points of $f_{2,\eta}$.

Let us now consider the equation  $P_{2,\eta}(x_2)=0$ with $x_2$ complex.  Since the polynomial has real coefficients, then the complex roots occur in conjugate pairs. 	
All solutions are also solutions of the equation $P_{2,\eta}(\ov{x}_2)=0$. Taking $z_2=x_2+i\cdot 0$, these points correspond to solutions of $f_2(z_2) =\ov{z_2}$. Hence a pair of complex conjugate roots $(w,\ov{w})$ of $P_{2,\eta}$ corresponds to a periodic orbit of period $2$ of $f_2$. 	
 
Let us consider $g_{2,\eta} =f_{2,\eta}\circ f_{2,\eta}$. Then $g_{2,\eta}$ is a  $k$ real parameter unfolding of a codimension $k$ holomorphic parabolic germ, which always has $k+1$ fixed points counting multiplicities. The equation for fixed points of $g_{2,\eta}$ is given by a  Weierstrass polynomial $p_{\eta}(z_2)$ depending real-analytically on $\eta$. The fixed points of $g_{2,\eta}$ are either fixed points of $f_{2,\eta}$ or belong to pairs $(w,\ov{w})$ of periodic points of $f_{2,\eta}$ with period $2$. Hence $p_\eta$ has real coefficients when $\eta$ is real. It follows that  $p_\eta\equiv P_{2,\eta}$.

\bigskip Let us now write $g_{2,\eta}$ in the form
$$g_{2,\eta}(z_2)= z_2+ P_{2,\eta}(z_2) \left(1+ q_\eta (z_2) + P_{2,\eta}(z_2) H_\eta(z_2)\right).$$  
Let $w_1, \dots, w_{k+1}$ be the fixed points of $g_{2,\eta}$. 
There exists a polynomial $S_\eta(z_2)$ de degree at most $k$ such that 
$$\log\left(g_{2,\eta}'(w_j)\right)= P_{2,\eta}'(w_j) (1+S_\eta(w_j)).$$
Indeed, when  the $w_j$ are distinct, let $M_j:=\frac{\log\left(g_{2,\eta}'(w_j)\right)}{P_{2,\eta}'(w_j)}-1$. Then such  a polynomial $S_\eta(z_2)$ is found by the following Lagrange interpolation formula:
$$S_\eta(z_2)=-\frac{\left|\begin{array}{lllll}0&1&z_2&\dots&z_2^k\\
M_1&1&w_1&\dots&w_1^k\\
\vdots&\vdots&\vdots&\vdots&\vdots\\
M_{k+1}&1&w_{k+1}&\dots&w_{k+1}^k\end{array}\right|}
{\left|\begin{array}{llll} 1&w_1&\dots&w_1^k\\
\vdots&\vdots&\vdots&\vdots\\
1&w_{k+1}&\dots&w_{k+1}^k\end{array}\right|}.$$
$S_\eta$ depends analytically on $\eta$, since it is invariant under permutations of the $w_j$. Moreover,  limits exist when two fixed points coallesce. Extending $\eta$ to complex values and using Hartogs' theorem allows to conclude that limits exist when more than two fixed points coallesce.
Since $P_{2,\eta}$ has real coefficients, since the complex conjugate roots of $P_{2,\eta}$ correspond to periodic points of period 2 of $f_{2,\eta}$ and using Lemma~\ref{lemma:multipliers}, it follows that for each root $w_j$ of $P_{2,\eta}$, then $\ov{w}_j$ is a root of $P_{2,\eta}$ and $\ov{M}_j:=\frac{\log\left(g_{2,\eta}'(\ov{w}_j)\right)}{P_{2,\eta}'(\ov{w}_j)}-1$, and thus that $S_\eta$ has real coefficients. 

Hence the logarithms of the multipliers at the fixed points of $g_{2,\eta}$ are the eigenvalues at the singular points of the vector field
\begin{equation}\dot z_2=v_{\eta}(z_2) = P_{2,\eta}(z_2)(1+S_\eta(z_2)).\label{pol_vf}\end{equation} By the variant of Kostov's theorem valid for real analytic dependence on parameters \cite{KR20}, there exists exactly $k$ changes of coordinate and parameter $(z_2,\eta)\mapsto (z_3,\eps)$ transforming \eqref{pol_vf} to 
$$\dot z_3 = \frac{z_3^{k+1} +\eps_{k-1} z_3^{k-1} +\dots + \eps_1z_3+\eps_0}{1+b(\eps)z_3^k}:=\frac{P_\eps(z_3)}{1+b(\eps)z_3^k}.$$
The $k$ one-parameter families of changes of coordinates are obtained one from another using the action of the rotation group of order $k$ on that vector field
$$(z_3, \eps_{k-1}, \dots, \eps_1,\eps_0)\mapsto (\tau z_3, \tau^{-k+2}\eps_{k-1}, \dots, \eps_1,\tau \eps_0),$$
where $\tau^k=1$. The one tangent to the identity, preserves the real axis, which is a privileged direction for $f_{3,\eta}$ (i.e. $f_{2,\eta}$ in the $z_3$ variable). Hence we choose a change of coordinate tangent to the identity (changes $z_3\mapsto -z_3$ are also allowed when $k$ is even). 

At this step the map $g_\eps$ is prepared. But the map $f_{3,\eps}$ may not be prepared yet. Indeed the derivatives of $f_{3,\eps}$ are not intrinsic. 
Considering that solutions of $P_\eps(z_3)=0$ are also solutions of $P_\eps(\ov{z}_3)=0$ and that these solutions are solutions of  $f_{3,\eps}(z_3)=\ov{z}_3$, then $f_{3,\eps}$ has the form
$$f_{3,\eps}(z_3)=\ov{z}_3+P_\eps(\ov{z}_3)M(\eps,\ov{z}_3).$$ 
By further dividing $M-\frac12$ by $P_\eps$, namely $$M(\eps,z_3)= \frac12+\sum_{\ell=0}^k m_\ell(\eps) z_3^\ell +P_\eps(z_3) N_\eps(z_3), $$ this yields
$$f_{3,\eps}(z_3)=\ov{z}_3+P_\eps(\ov{z}_3)\left( \frac12+\sum_{\ell=0}^k m_\ell(\eps) \ov{z}_3^\ell +P_\eps(\ov{z}_3) N_\eps(\ov{z}_3)\right).$$ 
If $w_1, \dots, w_{k+1}$ are the solutions of $P_\eps(z_3)=0$, then 
$$f_{3,\eps}'(w_j)= 1+P_\eps'(\ov{w}_j) \left( \frac12+\sum_{\ell=0}^k m_\ell(\eps) \ov{w}_j^\ell \right).$$ 
By Lemma~\ref{lemma:multipliers}, we already know that 
$$f_{3,\eps}'(w_j)\ov{f_{3,\eps}'(\ov{w}_j)}\in\R.$$

We look for a  change of coordinate $Z=u_\eps(z_3) = z_3+P_\eps(z_3)\left(\sum_{\ell=0}^k D_\ell z_3^\ell\right)$ preserving the fixed points and periodic points of period $2$ of $f_{3,\eps}$ so that if $F_\eps=u_\eps\circ f_{3,\eps}\circ u_\eps^{-1}$, then  
\begin{equation}F_\eps'(w_j)=\ov{F_\eps'(\ov{w}_j)}, \qquad j=1, \dots, k+1.\label{property:F}\end{equation}
Note that $F_\eps(w_j)= \ov{w}_j$. Hence $$F_\eps'(w_j)= \frac{u_\eps'(\ov{w}_j)}{\ov{u_\eps'(w_j)}} f_{3,\eps}'(w_j).$$
Hence we ask that 
\begin{equation}u_\eps'(w_j)= \ov{\sqrt{f_{3,\eps}'(w_j)}}.\label{def:u}\end{equation} 
If $w_j\in\R$ is a fixed point of $f_{3,\eps}$, then $F_\eps'(w_j)=|f_{3,\eps}'(w_j)|\in \R_{\geq0}$.
If $(w_j,\ov{w}_j)$ is a periodic orbit of period 2, then $F_\eps'(w_j) =\ov{\sqrt{f_{3,\eps}'(\ov{w}_j)}}\sqrt{f_{3,\eps}'(w_j)}$ and $F_\eps'(\ov{w}_j) =\sqrt{f_{3,\eps}'(\ov{w}_j)}\ov{\sqrt{f_{3,\eps}'(w_j)}}$. Hence $F$ satisfies \eqref{property:F}.

We now need to prove that it is possible to construct $u$ mix analytic satisfying \eqref{def:u}.

Let $K_\eps(z_3)=\sum_{\ell=0}^k D_\ell z_3^\ell$. Then $u_\eps'(w_j)= 1+P_\eps'(w_j)K_\eps(w_j)$, while  
$$\ov{\sqrt{f_{3,\eps}'(w_j)}}= \sqrt{1+P_\eps'(w_j)\left(\frac12+\sum_{\ell=0}^k \ov{m_\ell}(\eps) w_j^\ell\right) }:=1+P_\eps'(w_j) \left( \frac14+V_\eps(w_j)\right),$$
for some analytic function $V_\eps$.
Hence $K_\eps(w_j)=\frac14 +V_\eps(w_j):=W_j$. For distinct $w_j$ the polynomial $K_\eps$ is given by a Lagrange interpolation formula 
$$K_\eps(z_3)=-\frac{\left|\begin{array}{lllll}0&1&z_3&\dots&z_3^k\\
W_1&1&w_1&\dots&w_1^k\\
\vdots&\vdots&\vdots&\vdots&\vdots\\
W_{k+1}&1&w_{k+1}&\dots&w_{k+1}^k\end{array}\right|}
{\left|\begin{array}{llll} 1&w_1&\dots&w_1^k\\
\vdots&\vdots&\vdots&\vdots\\
1&w_{k+1}&\dots&w_{k+1}^k\end{array}\right|}.$$ Note that the conditions defining $K_\eps$ are analytic in $\eps$. Hence it is possible to complexify $\eps$. The formula has a limit when two $w_j$ coallesce. The limit also exists for the more degenerate cases by Hartogs' Theorem. Since the conditions are invariant under permutations of the $w_j$, the polynomial $K_\eps$ depends analytically on
$\eps$ by the symmetric function theorem. \end{proof}

\begin{corollary}\label{cor:canonical} When $k$ is odd, the canonical parameter of the prepared $f_\eps$ is unique. When  $k$ is even, conjugating $f_\eps$ with $L_{-1}(z) = -z$ yields a second prepared form $\hat{f}_{\hat{\eps}}= L_{-1}\circ f_\eps\circ L_{-1}$ with canonical parameter 
\begin{equation}  \hat{\eps}= (\eps_{k-1}, -\eps_{k-2}, \dots, \eps_1, -\eps_0).\label{canonical_odd}\end{equation}\end{corollary}

\section{Modulus of analytic classification}

We now consider a germ of generic antiholomorphic family unfolding a parabolic point of codimension $k$ in prepared form
\begin{equation}f_\eps(z)=\ov{z}+ P_\eps(\ov{z})\left(\frac12+Q_\eps(\ov{z})+P_\eps(\ov{z})R_\eps(\ov{z})\right),\label{family_prepared}\end{equation} as described in Theorem~\ref{thm_prepared}.
As in Lemma~\ref{lem:complexify}, we complexify the parameter $\eps$ in $(\C^k,0)$, we ask that $f_\eps$ depends antiholomorphically on $\eps$, and we define  the second iterate 
as in \eqref{def:g}. Germs of generic analytic unfoldings of a holomorphic parabolic point of codimension $k$ have been studied in \cite{R15}  and we will see that two prepared germs of antiholomorphic families $f_{1,\eps}$ and $f_{2,\eps}$ are conjugate under a conjugacy tangent to the identity depending real-analytically on $\eps\in (\R^k,0)$  if and only if the corresponding homolorphic  families $g_{1,\eps}= f_{1,\epsbar}\circ f_{1,\eps}$ and $g_{2,\eps}= f_{2,\epsbar}\circ f_{2,\eps}$, with complex analytic dependence on $\eps\in (\C^k,0)$, are analytically conjugate under a conjugacy tangent to the identity.

For real $\eps$, the formal normal form of $f_\eps$ is given by $\sigma\circ v_\eps^{\frac12}=v_{\eps}^{\frac12}\circ \sigma$, where $v_\eps^{t}$ is the time $t$ of the vector field
\begin{equation} v_\eps= \frac{P_\eps(z)}{1+b(\eps)z^k}\,\frac{\partial}{\partial z}, \label{def:v}\end{equation} and
\begin{equation}
 P_\eps(z)= z^{k+1} + \sum_{j=0}^{k-1} \eps_jz^j.\label{def:P} \end{equation}
 For complex values of $\eps$ we have to think of the formal normal form meaning that  \begin{equation}\hat{h}_{\ov{\eps}}\circ f_\eps\circ (\hat{h}_\eps)^{-1}= \sigma\circ v_\eps^{\frac12}=v_{\ov{\eps}}^{\frac12}\circ \sigma\label{change_nf} \end{equation} for some formal map $\hat{h}_\eps$.

We want to describe the dynamics of the germ of family. In practice, this means describing the dynamics for $z$ in a disk $\D_r$ of radius $r$, for all values of the parameter in some polydisk $|\eps|<\rho$. 
The general spirit is that if $\rho$ is taken sufficiently small so that the fixed points stay bounded away from $\partial \D_r$, for instance in $\D_{r/2}$, then the dynamics is structurally stable in the neighborhood of $\partial \D_r$, and this dynamics organizes  the whole dynamics inside the disk. The modulus of analytic classification measures the obstruction to transforming analytically the family into the formal normal form. To construct the modulus, we transform the family almost uniquely to the normal form on (generalized) sectors in $z$-space. (Note that $f_\eps$ sends one sector to a different sector.) In accordance with the general spirit just mentioned, these generalized sectors are constructed from the behavior around $\partial \D_r$ and then following the dynamics inwards. Then the modulus is given by the mismatch of the normalizing transformations. In the construction,  $2k$ generalized sectors are needed, if we add the additional constraint that the generalized sectors have a limit when $\eps\to 0$.

In practice, it is more natural to change coordinate to the time coordinate of the vector field $v_\eps$, given by $$Z_\eps=\int\frac{1+b(\eps)z^k}{P_\eps(z)}\,dz.$$ In this new coordinate $f_\eps$ is transformed to $F_\eps= Z_{\ov{\eps}}\circ f_\eps\circ Z_\eps^{-1}$ and  the normal form to $T_{\frac12}\circ \Sigma,$ where $\Sigma$ is a complex conjugation defined in the Riemann surface of the time coordinate by lifting $\sigma$ (see Definition~\ref{def:Sigma} below), and  $T_{\frac12}$ is the translation by $\frac12$ (see Notation~\ref{notation}). Then, in the $Z_\eps$-coordinate, the sectors will correspond to the saturation by the dynamics of strips transversal to the horizontal direction and we need to consider pairs of sectors for $Z_\eps$ and $Z_{\ov{\eps}}$.  

\subsection{The time coordinate $Z_\eps$}

The time coordinate $Z_\eps$ is multivalued over the disk punctured at the fixed points and the image $Z_\eps\left(\D_r\setminus\{P_\eps(z)=0\}\right)$ is a complicated Riemann surface. In practice we work with  $2k$ charts defined from $\partial \D_r$ and going inwards. 
For $j=0,\pm 1, \dots \pm k$ (with indices $({\rm mod}\: 2k))$, we define 
$$Z_{\eps,j}(z) = \int_{\zeta_j}^z \frac{1+b(\eps)z^k}{P_\eps(z)} \,dz,$$
where $\zeta_0=r$ and, for $j=\pm1, \dots, \pm k$, $\zeta_j$ close to $\partial \D_r$ is defined by $\int_{\gamma_j} \frac{1+b(\eps)z^k}{P_\eps(z)} \,dz= \frac{2\pi i b(\eps)}{k}$ with $\gamma_j$ an arc from $\zeta_{j-1}$ to $\zeta_j$ located in the neighborhood of $\partial \D_r$. The chart for $Z_{\eps,j}$ contains the arc $\left\{re^{i\theta}\mid \theta\in (\frac{\pi j}{k}-\frac{\pi}{2k}, \frac{\pi j}{k}+\frac{\pi}{2k})\right\}$.
In particular
\begin{equation} 
Z_{\eps,j}(z) = Z_{\eps,j-1}(z) -  \frac{2\pi i b(\eps)}{k}\label{relation:Z_j}\end{equation}
where the indices are $({\rm mod}\: 2k)$.

Each simple singular point $z_s$ of $v_\eps$ has a nonzero period given by $2\pi i\, {\rm Res} \left(\frac{1+b(\eps)z^k}{P_\eps(z)}, \allowbreak z_s\right)$. Moreover, the fixed points of $f_\eps$ are sent at infinity in directions which rotate when the parameter varies. Note that the periods of points are unbounded and have an infinite limit when two singular points merge together. 

What is important is that the whole dynamics is organized by the structurally stable behavior in the neighborhood of $\partial D_r$ (see Figure~\ref{fleur3}). 
For sufficiently small $\eps$ the image of $\partial D_r$ is, roughly speaking, a $k$-covering of a curve close to a circle of radius $R=\frac1{kr^k}$ (there is an extra discrepancy of $2\pi i b(\eps)$, which is small compared to the radius $R$) and the interior of the disk is sent to a $k$-sheeted surface on the exterior of the image circle (but there is again an extra discrepancy of $2\pi i b(\eps)$). The interior of the image circle is often called a \emph{hole}.
Because of the periods, there are sequences of holes on the Riemann surface of $Z_\eps$. In the limit $\eps=0$, only one hole remains, the \emph{principal hole}, while the others have disappeared at infinity. 

\begin{definition}\label{def:Sigma} The complex conjugation $\sigma$ is lifted in the time coordinate to $\Sigma$. For real $\eps$, then $\Sigma$ is the usual complex conjugation in the coordinate $Z_{0,\eps}$, and then extended antiholomorphically over the Riemann surface of the time. It is then antiholomorphically extended in nonreal $\eps$.
If $U_{\eps,j}$ is the image of $Z_{\eps,j}$ then  $\Sigma: U_{\eps,j}\rightarrow U_{\ov{\eps},-j}$ satisfies
$$\Sigma \circ Z_{\eps,j}= Z_{\ov{\eps},-j}\circ \sigma.$$
\end{definition}

\subsection{The $2k$ sectors in $z$-space}

\begin{figure} \begin{center} 
\subfigure[In $z$-space]{\includegraphics[width=5cm]{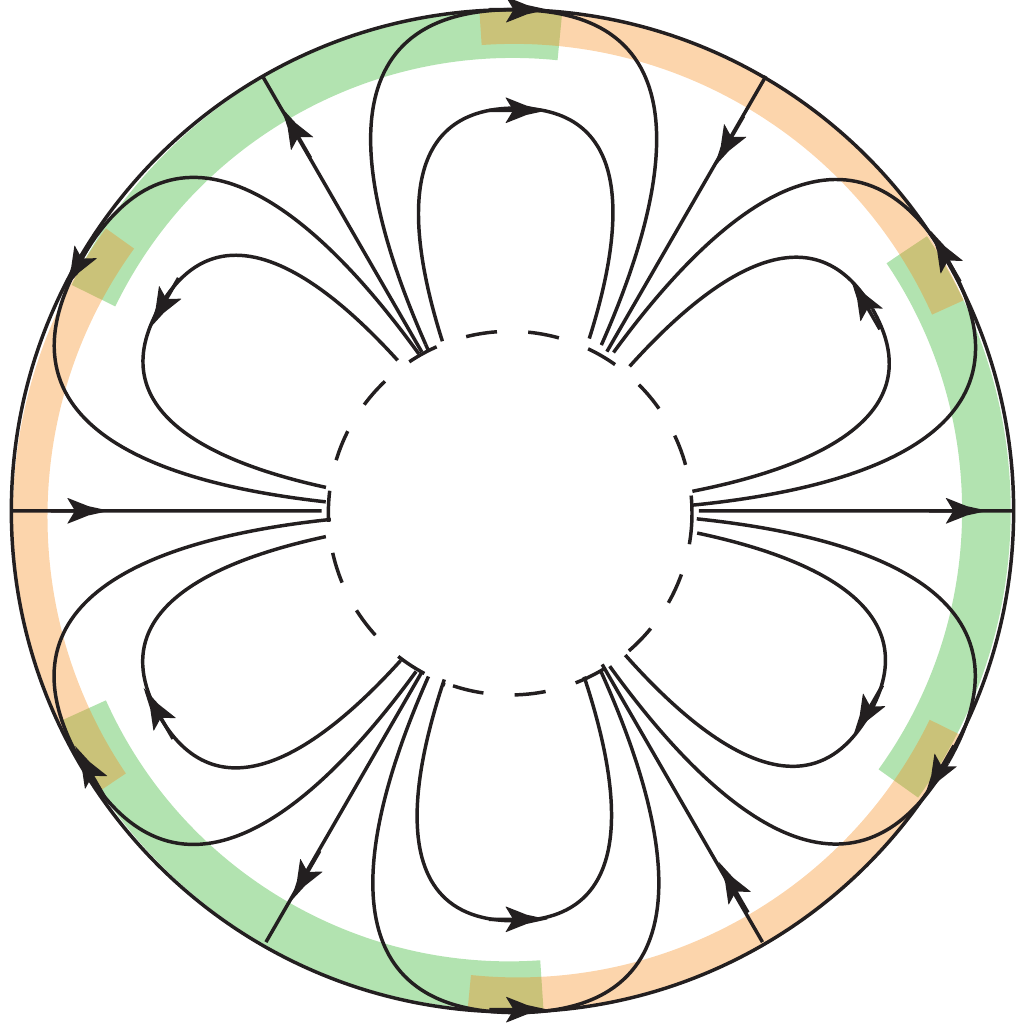}}\qquad\quad\subfigure[In $Z_\eps$-space]{\includegraphics[width=5cm]{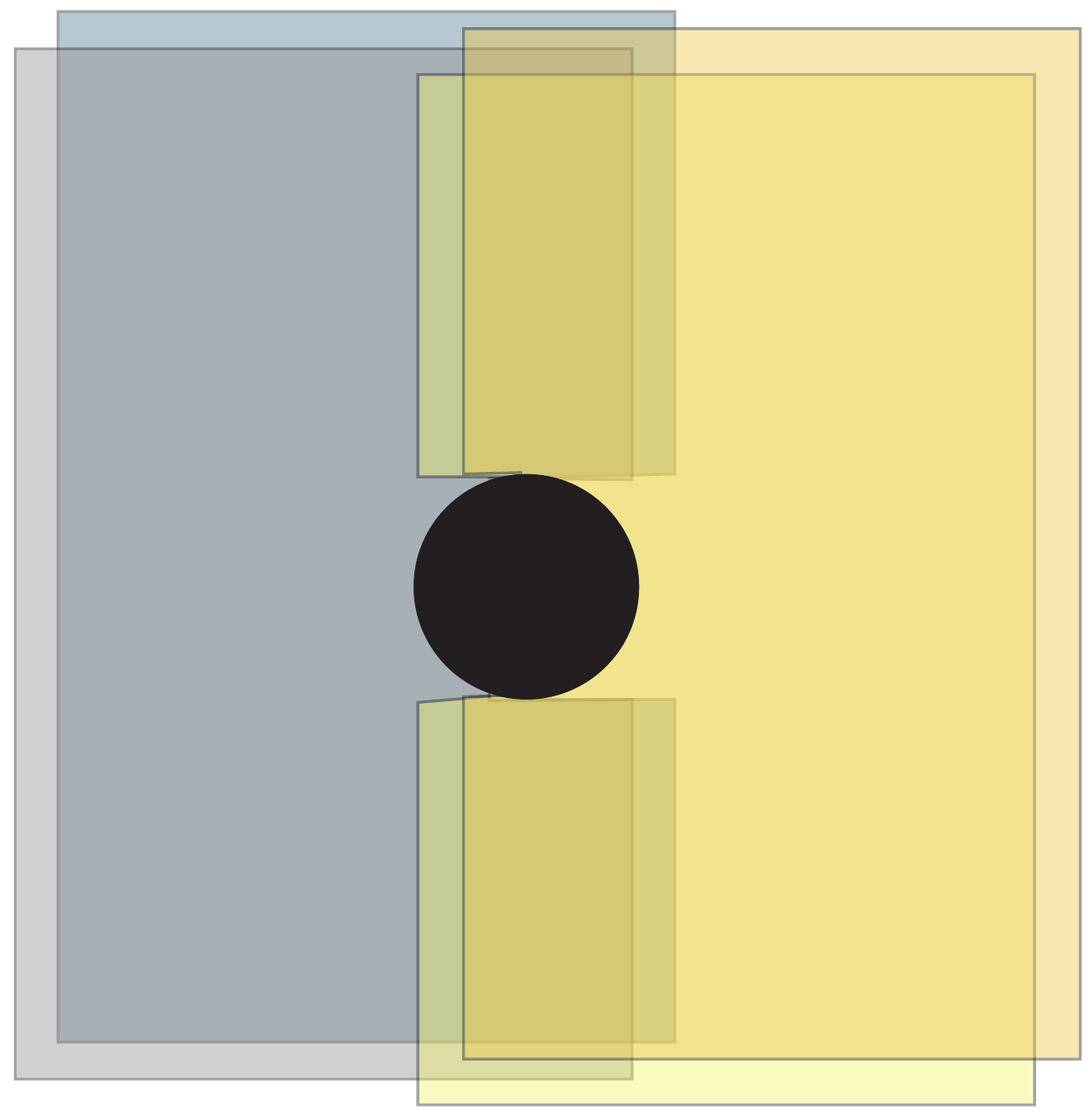}}\caption{The $2k$ sectors near  $\partial \D_r$ and the corresponding sectors in time space.}\label{fleur3}\end{center}\end{figure}
The $2k$ sectors in $z$-space will be attached to $\partial \D_r$ as in Figure~\ref{fleur3}. In the generic case of simple singular points their boundary will be given by (see Figure~\ref{ssecteurs}):
\begin{itemize} 
\item one arc $\gamma$ along $\partial \D_r$ containing $\left\{re^{i\theta}\mid \theta\in (\frac{\pi j}{k}-\frac{\pi}{2k}, \frac{\pi j}{k}+\frac{\pi}{2k})\right\}$ for some $j$ as in Figure~\ref{fleur3},
\item one arc from one end of $\gamma$ to one singular point, 
\item a second arc from the other end $\gamma$ to a second singular point,
\item an arc between the two singular points. 
\end{itemize} 
The last three arcs will often be spiralling when approaching the singular points. All together the $2k$ sectors  provide a covering of $\D_r\setminus\{P_\eps(z)=0\}$. 
\begin{figure} \begin{center} 
\subfigure{\includegraphics[width=2.8cm]{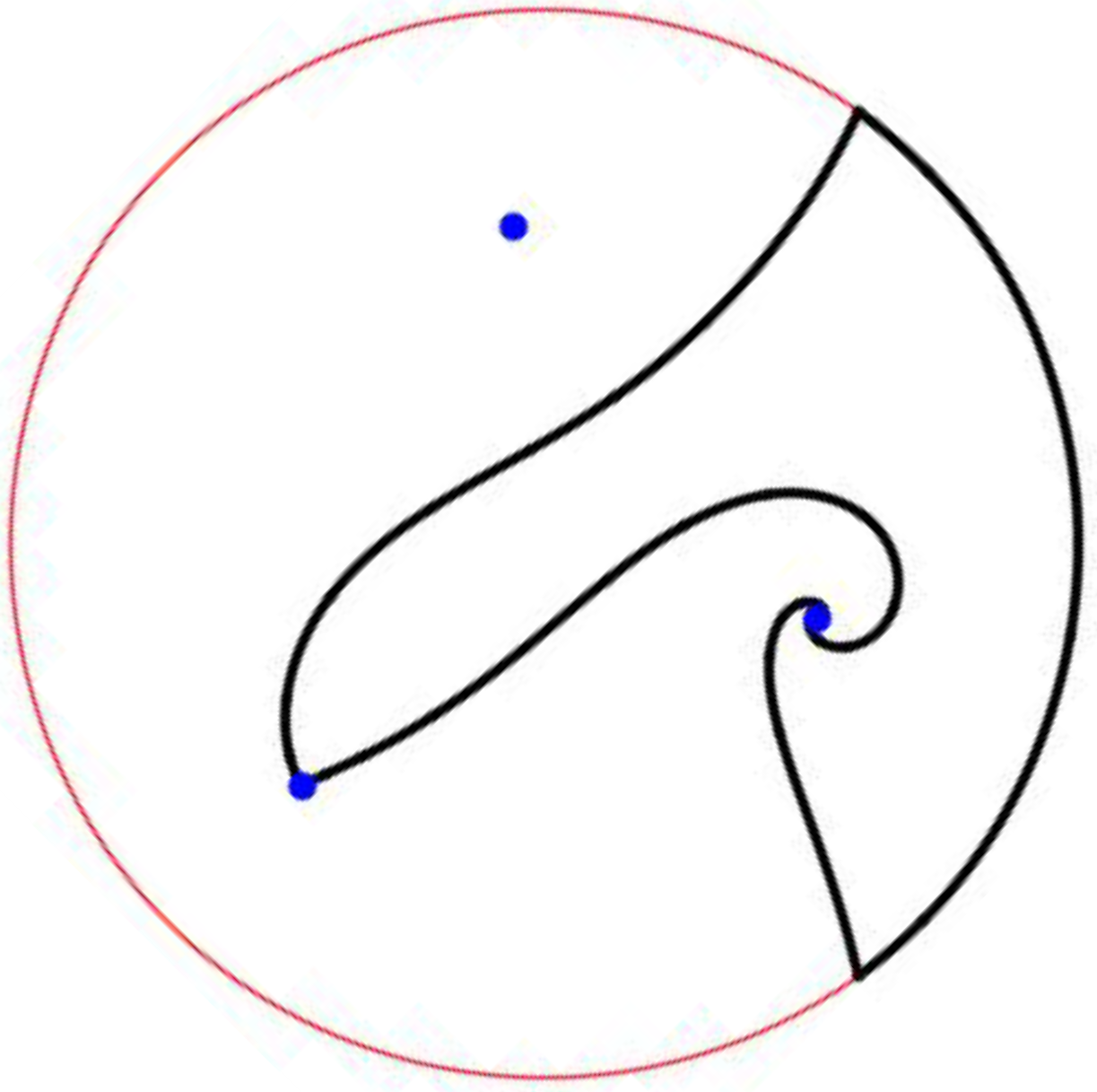}}\quad\subfigure{\includegraphics[width=2.8cm]{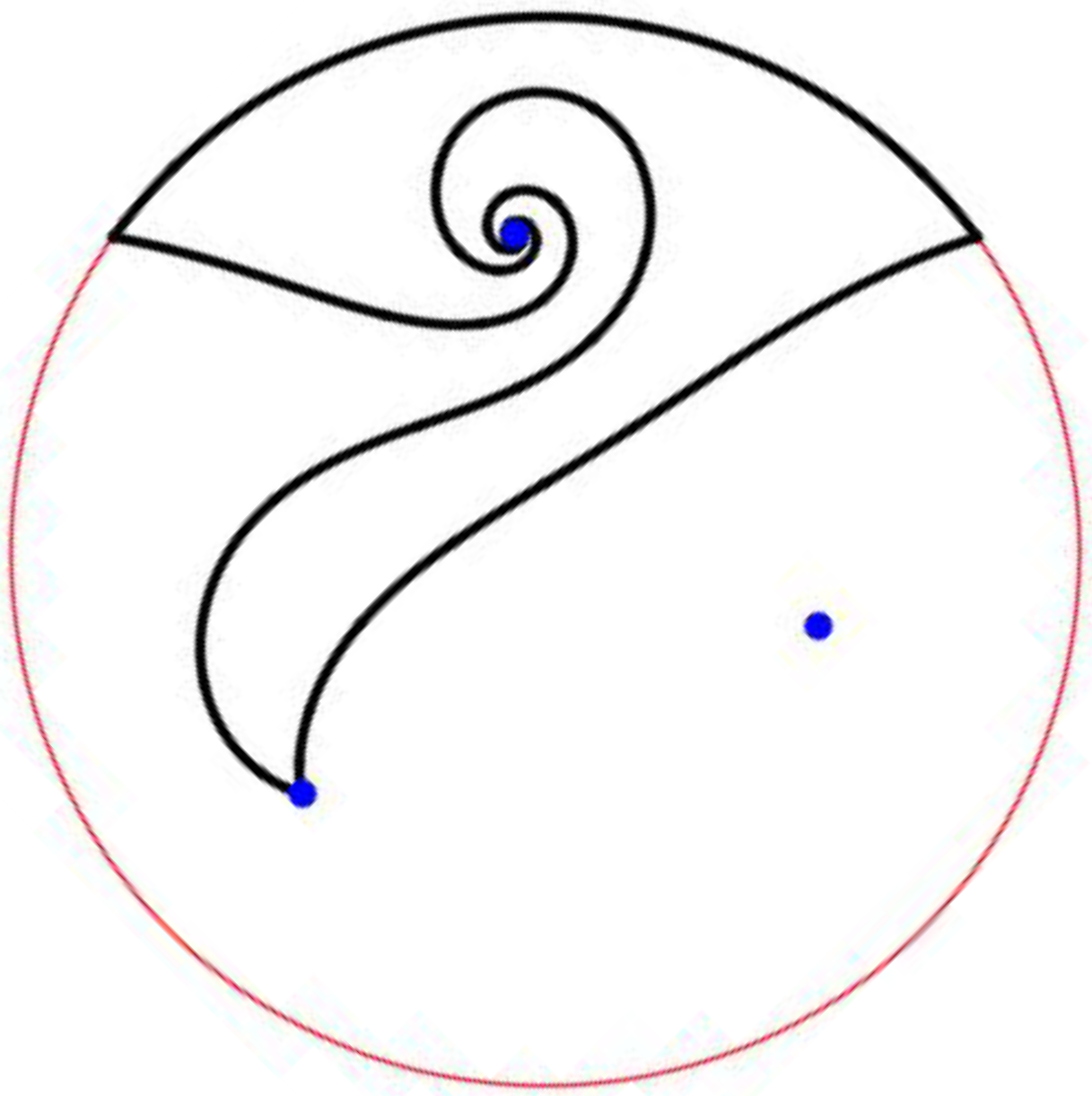}}\quad
\subfigure{\includegraphics[width=2.8cm]{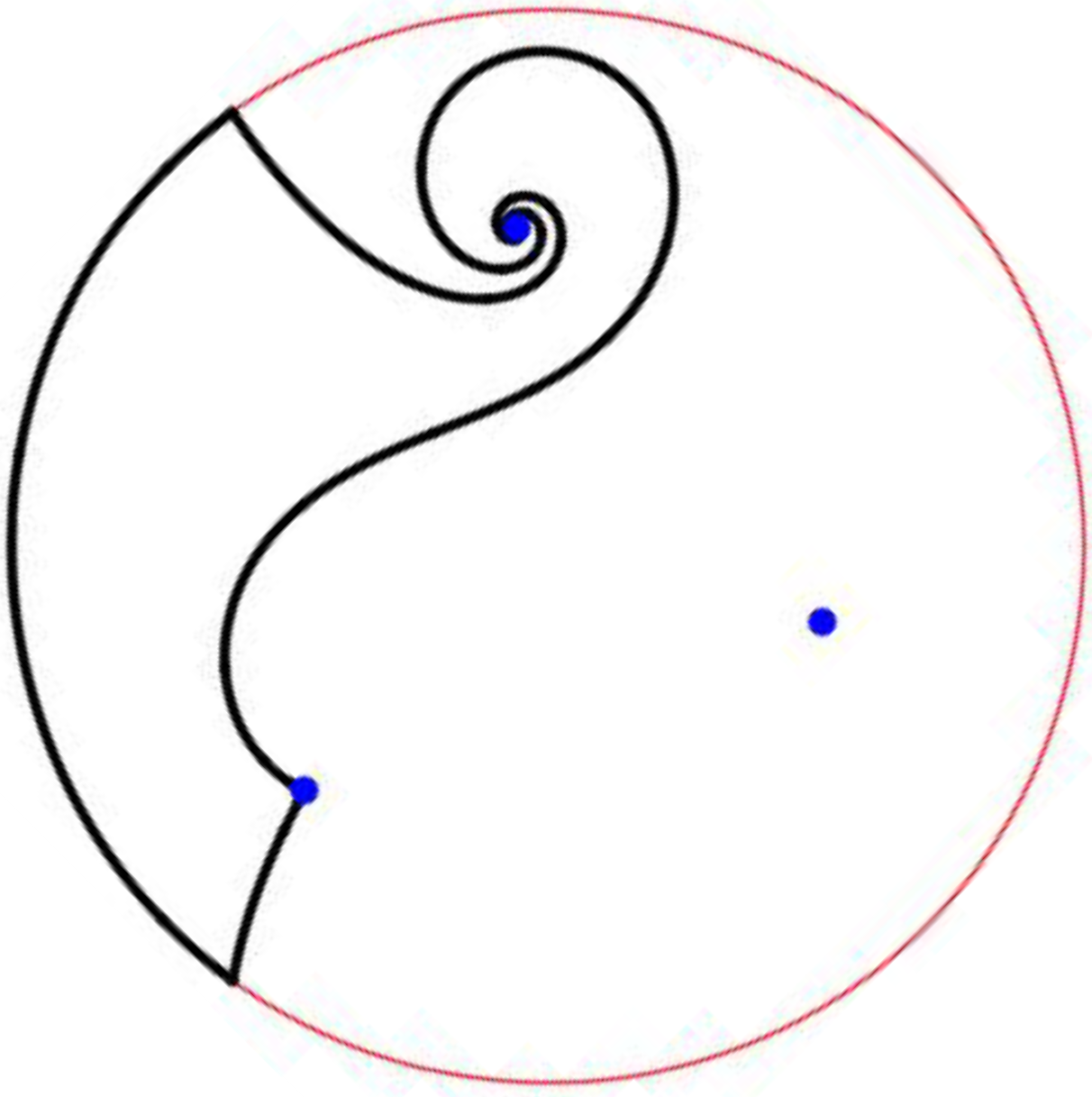}}\quad\subfigure{\includegraphics[width=2.8cm]{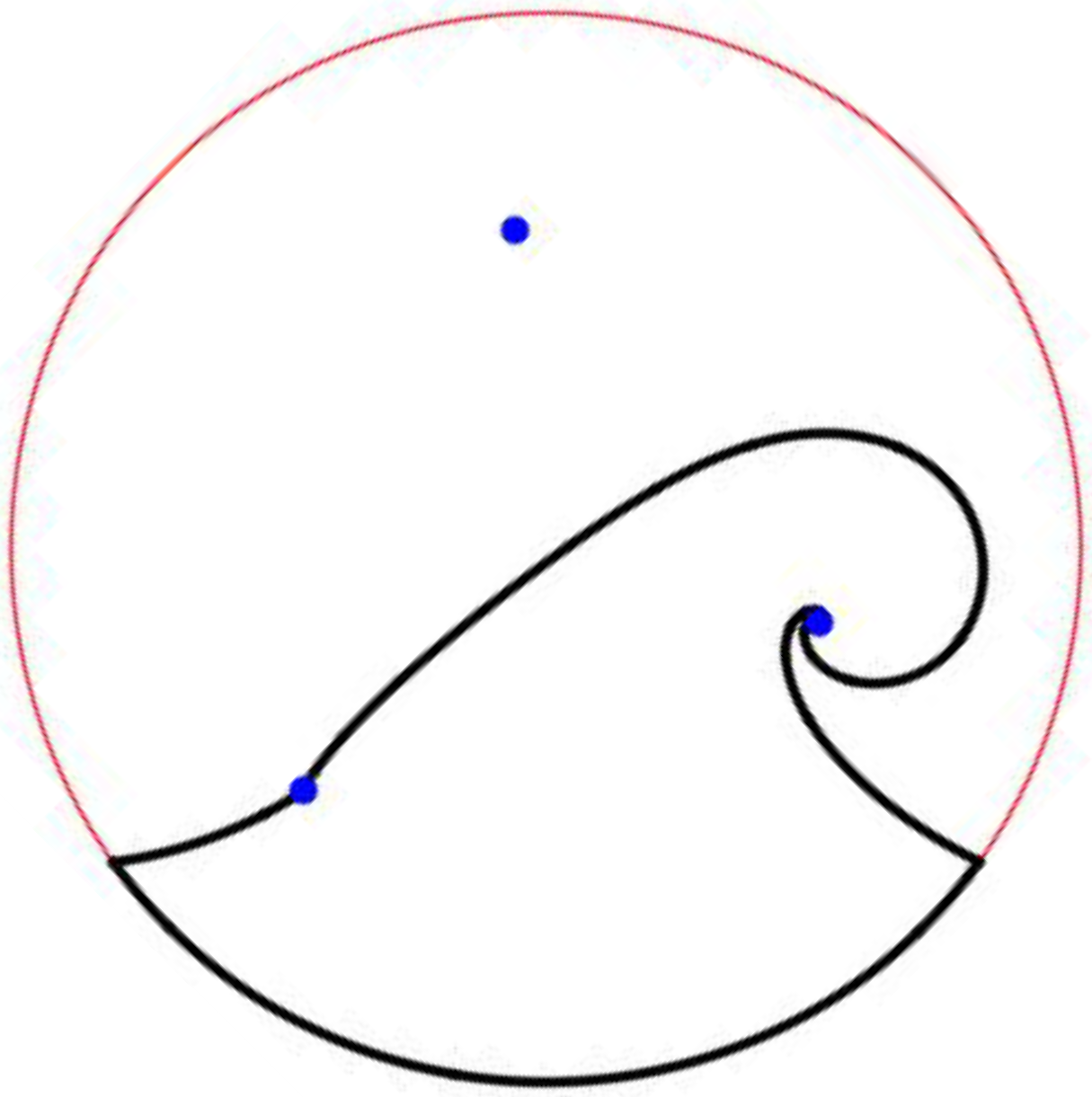}}\caption{The $4$ sectors for $P_\eps(z) =z^3+\frac{2+i}{20}z+\frac{1+6i}{30}e^{\frac{i\pi}4}$.}\label{ssecteurs}\end{center}\end{figure}
Note the shape of the intersection of the four sectors in Figure~\ref{intersecteurs}.
\begin{figure} \begin{center} 
\includegraphics[width=5cm]{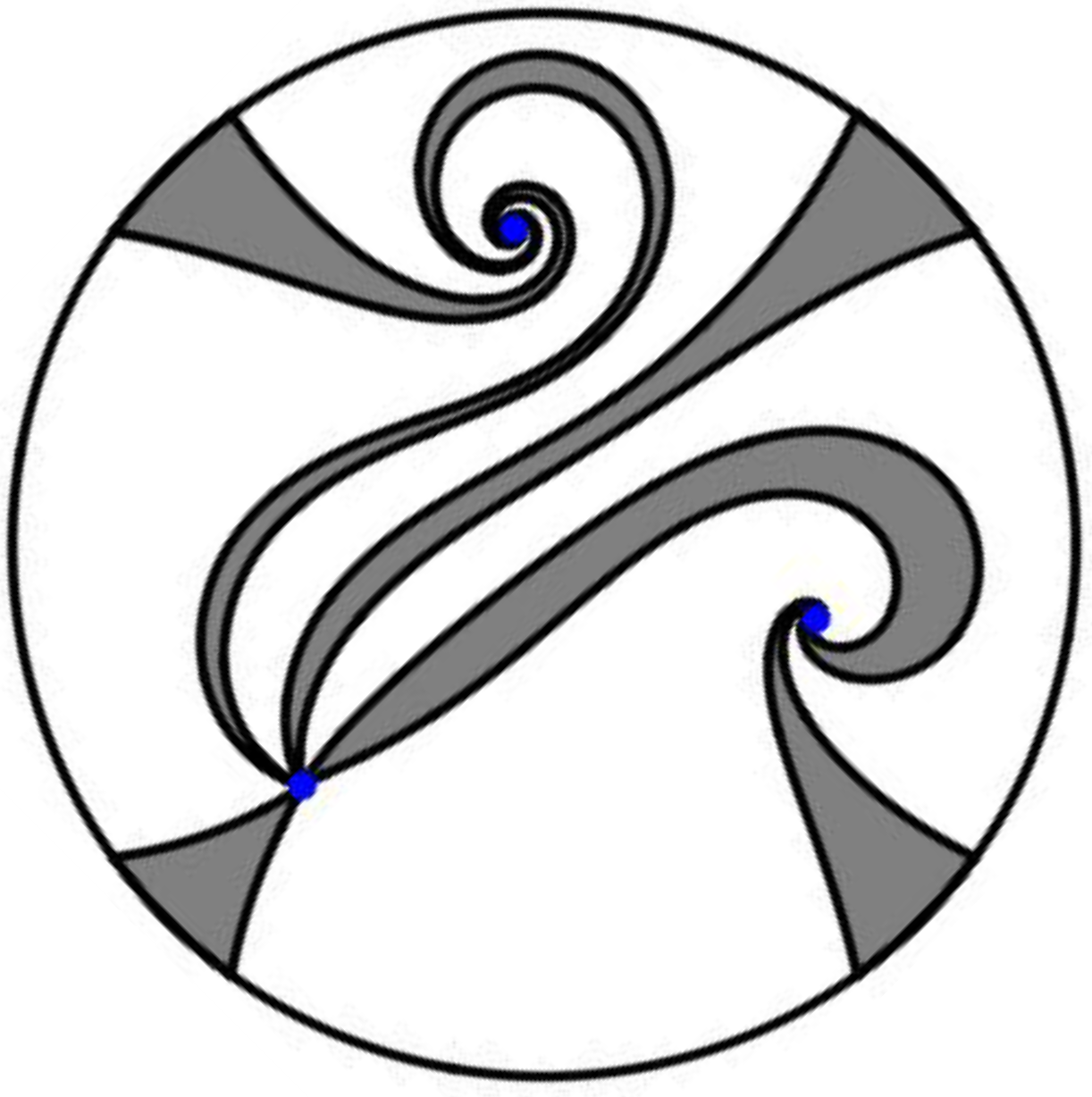}\caption{The intersections of the four sectors of Figure~\ref{ssecteurs}: four intersection parts link a fixed point to the boundary and have a limit when the fixed points merge together. The two other parts (called \emph{gate sectors}) link two fixed points and disappear when the two points merge together. }\label{intersecteurs}\end{center}\end{figure}

Because the singular points move around inside the disk, the  $2k$ sectors cannot be defined depending continuously on the parameters in a uniform way in the parameter space. Hence we will need to  use a covering of the parameter space minus the discriminant set (where multiple fixed points occur) by $C(k)=\frac{\binom{2k}{k}}{k+1}$ simply connected \emph{sectoral domains}. To describe these sectoral domains we need to consider the dynamics of $w_\eps= iv_\eps$. But, in practice, it suffices to work with the polynomial vector field $iP_\eps(z)\frac{\partial}{\partial z}$, which has the same fixed points as $w_\eps$ and whose real-time trajectories inside $\D_r$ are close to those of $w_\eps$. 

The \lq\lq generic\rq\rq\ polynomial vector fields have been described by Douady-Estrada-Sentenac \cite{DES05} (see Section~\ref{sec:DES} below).  
The sectoral domains are enlargements of the $C(k)$ generic strata of Douady-Estrada-Sentenac \cite{DES05}  and cover the parameter space minus the discriminant set. The discriminant set has complex codimension 1. Hence, to secure conjugacy of the families over the full parameter space, it will be sufficient to describe a modulus outside the discriminant set, thus  guaranteeing that two families with same modulus are conjugate over the complement of the discriminant set, and then to check  that the conjugacy remains bounded when approaching the discriminant set. 

\subsection{The work of Douady, Estrada and Sentenac}\label{sec:DES}
The paper \cite{DES05} classifies \lq\lq generic\rq\rq\  monic polynomial vector fields $P_\eps(z)\frac{\partial}{\partial z}$ up to affine transformations by means of an invariant composed of two parts: a combinatorial part and an analytic part given by a vector of $\H^k$. (The corresponding description for $iP_\eps(z)\frac{\partial}{\partial z}$ follows through $z\mapsto \tau z$ for $\tau^k=-i$.)

The dynamics of $P_\eps(z)\frac{\partial}{\partial z}$ is governed by the pole at infinity and its $2k$ separatrices alternately stable and unstable (see Figure~\ref{figure_infinity}). 
\begin{figure} \begin{center} 
\includegraphics[width=5cm]{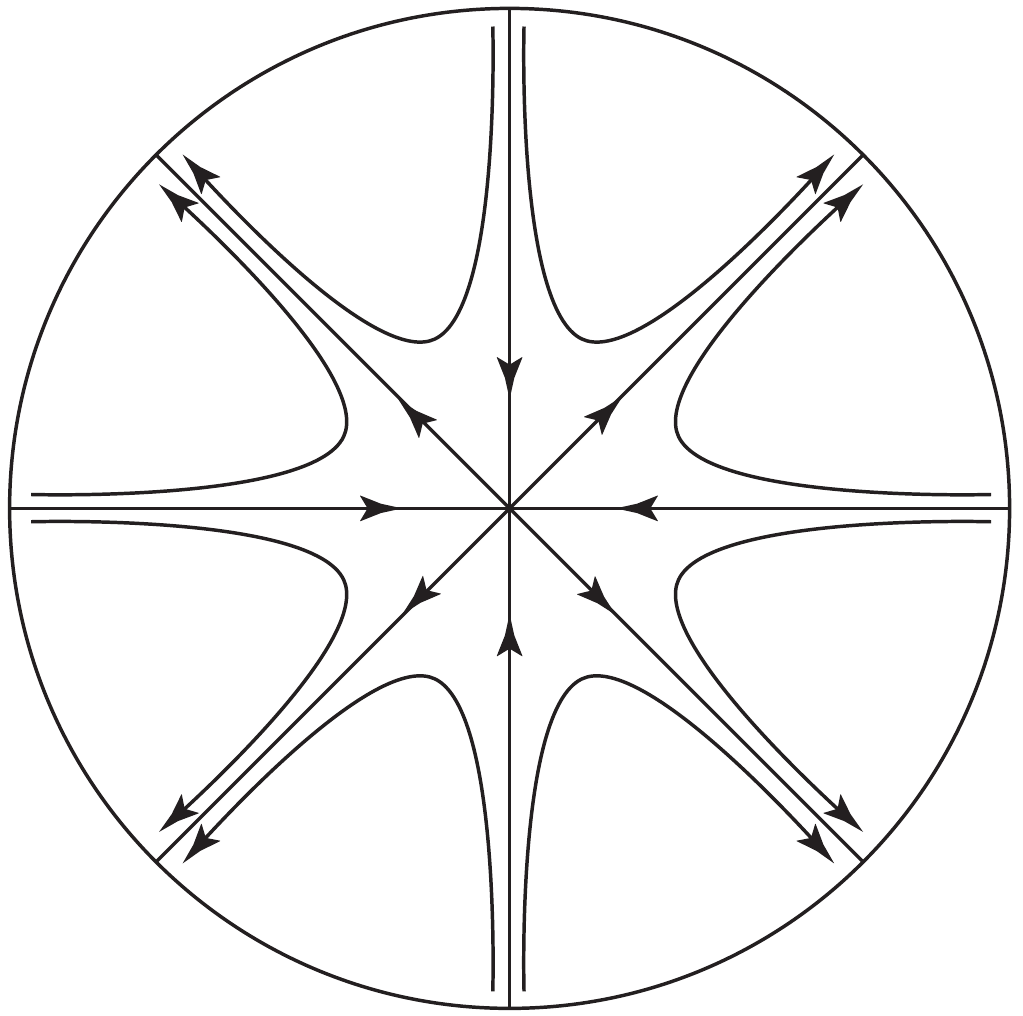}\caption{The pole at infinity of $P_\eps(z)\frac{\partial}{\partial z}$ and its separatrices organizing the dynamics in the  neighborhood of $\partial \D_r$ as in Figure~\ref{fleur3}.}\label{figure_infinity}\end{center}\end{figure}
Douady, Estrada and Sentenac
 have studied the generic case where the singular points are simple and there is no homoclinic loop through infinity, which we call \emph{DES-generic}. Under the  DES-generic hypothesis,   the separatrices land at the $k+1$ singular points, which are foci or nodes (the eigenvalue has a nonzero real part). Moreover, the singular points are linked by trajectories. Two  trajectories joining two singular points are called \emph{equivalent} if they have the same $\alpha$-limit and $\omega$-limit points. The equivalence classes of trajectories can be considered as the edges of a tree graph with $k+1$ vertices located at the fixed points. The combinatorial part of the  Douady-Estrada-Sentenac invariant is given by the tree graph and the way to attach it to the separatrices (see Figure~\ref{fig_skeleton}). There are $C(k)$ different combinatorial parts, yielding  $C(k)$ generic DES strata. Each DES stratum is parametrized by $\H^k$.

 \begin{figure}\begin{center}\includegraphics[width=6cm]{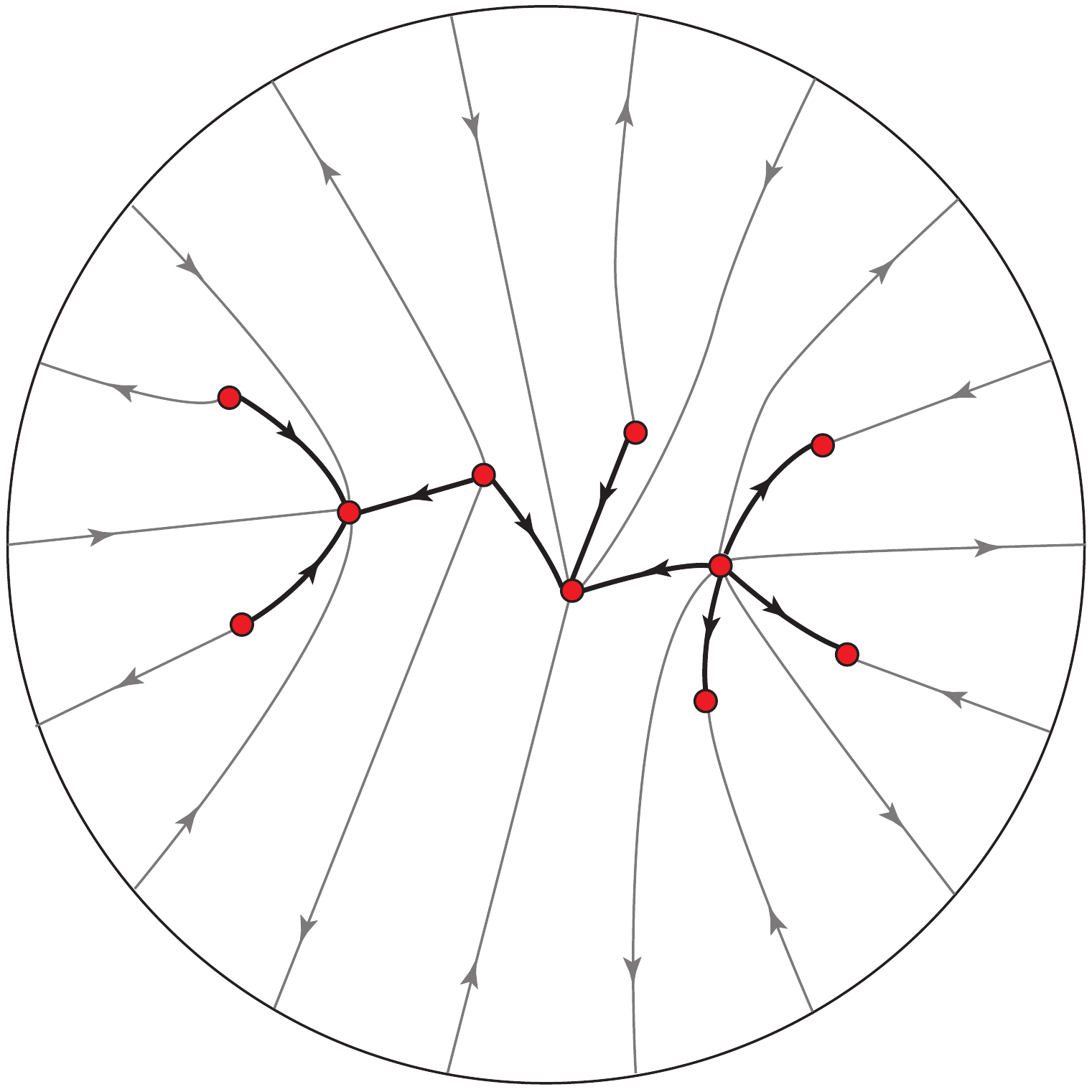}\caption{The tree graph and its attachment to the separatrices. (The figure is topological and the trajectories and separatrices could spiral when approaching the singular points.)}\label{fig_skeleton}\end{center}\end{figure}
  
Exceptionally, some separatrices can merge by pairs, one stable, one unstable, in homoclinic loops through $\infty$. A necessary condition for this to occur is that the sum of the periods of the singular points surrounded by the homoclinic loop is a real number. Generically, this occurs on hypersurfaces of real codimension 1, which separate the strata of DES-generic vector fields. 

Apart from the multiple singular points, the homoclinic loops are the only bifurcations. In particular, there are no limit cycles and any singular point with a pure imaginary eigenvalue is a center surrounded by a homoclinic loop through infinity. 

In the DES-generic case, the separatrices split the plane into $k$ connected regions, each adherent to two fixed points, one attracting, one repelling (see Figure~\ref{fig_skeleton_sectors}(a)). It is these connected regions for the vector field $iP_\eps(z)\frac{\partial}{\partial z}$ that will be used to define the $2k$ sectors. 
\begin{figure}\begin{center}\subfigure[Two connected regions]{\includegraphics[width=3.3cm]{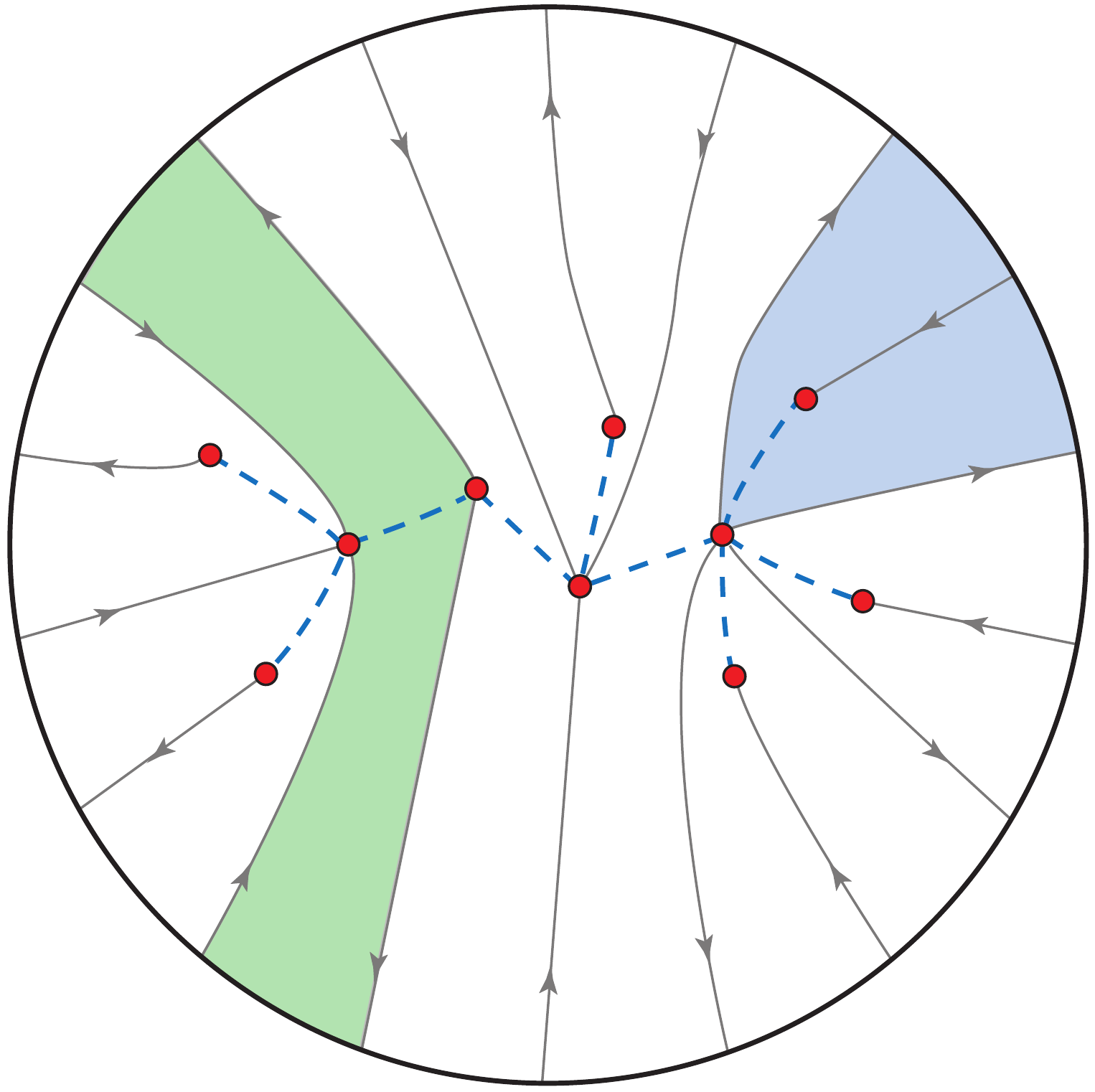}}\quad\subfigure[The corresponding half-regions]{\includegraphics[width=3.3cm]{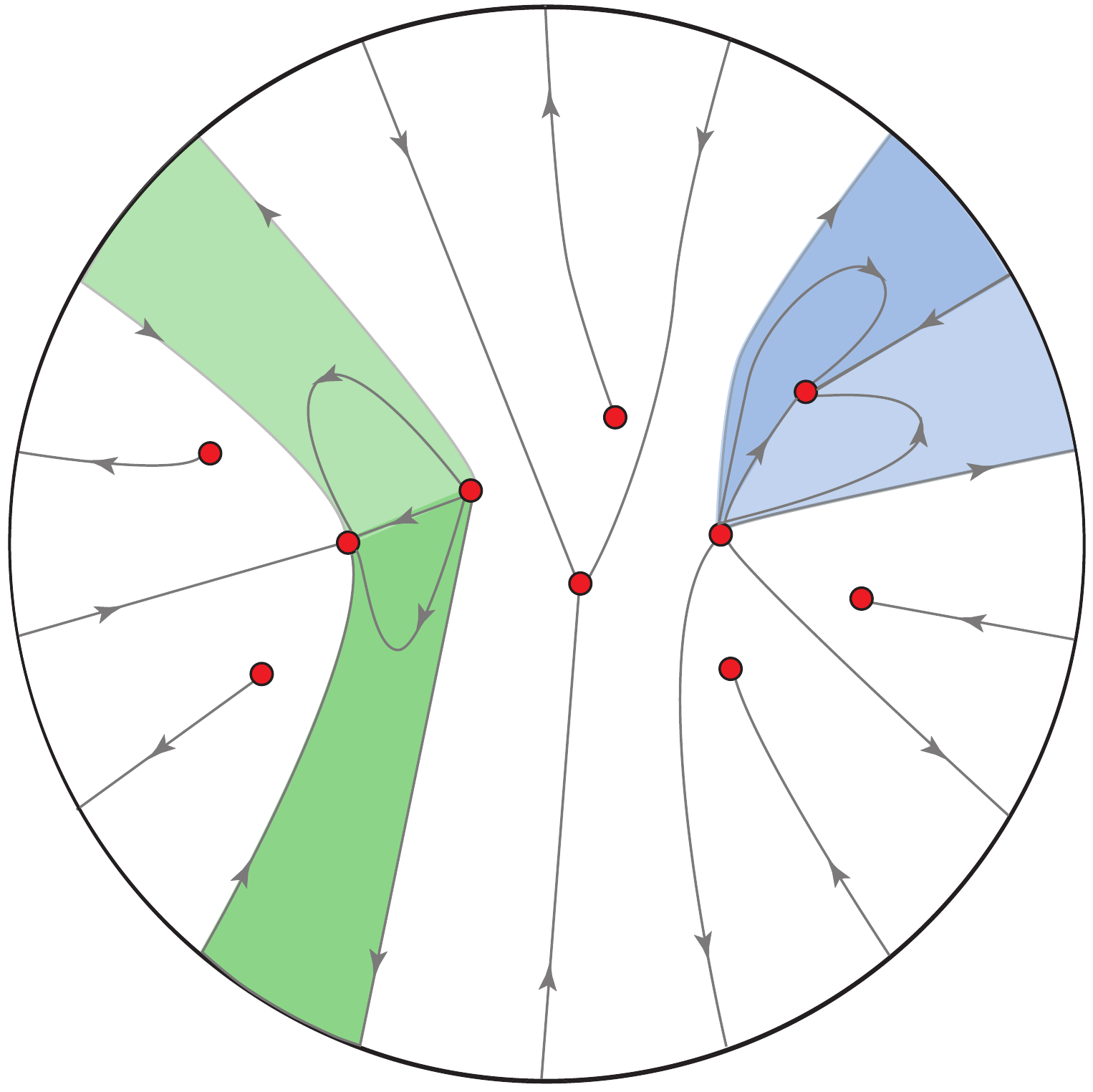}}\quad\subfigure[The enlarged half-regions]{\includegraphics[width=3.3cm]{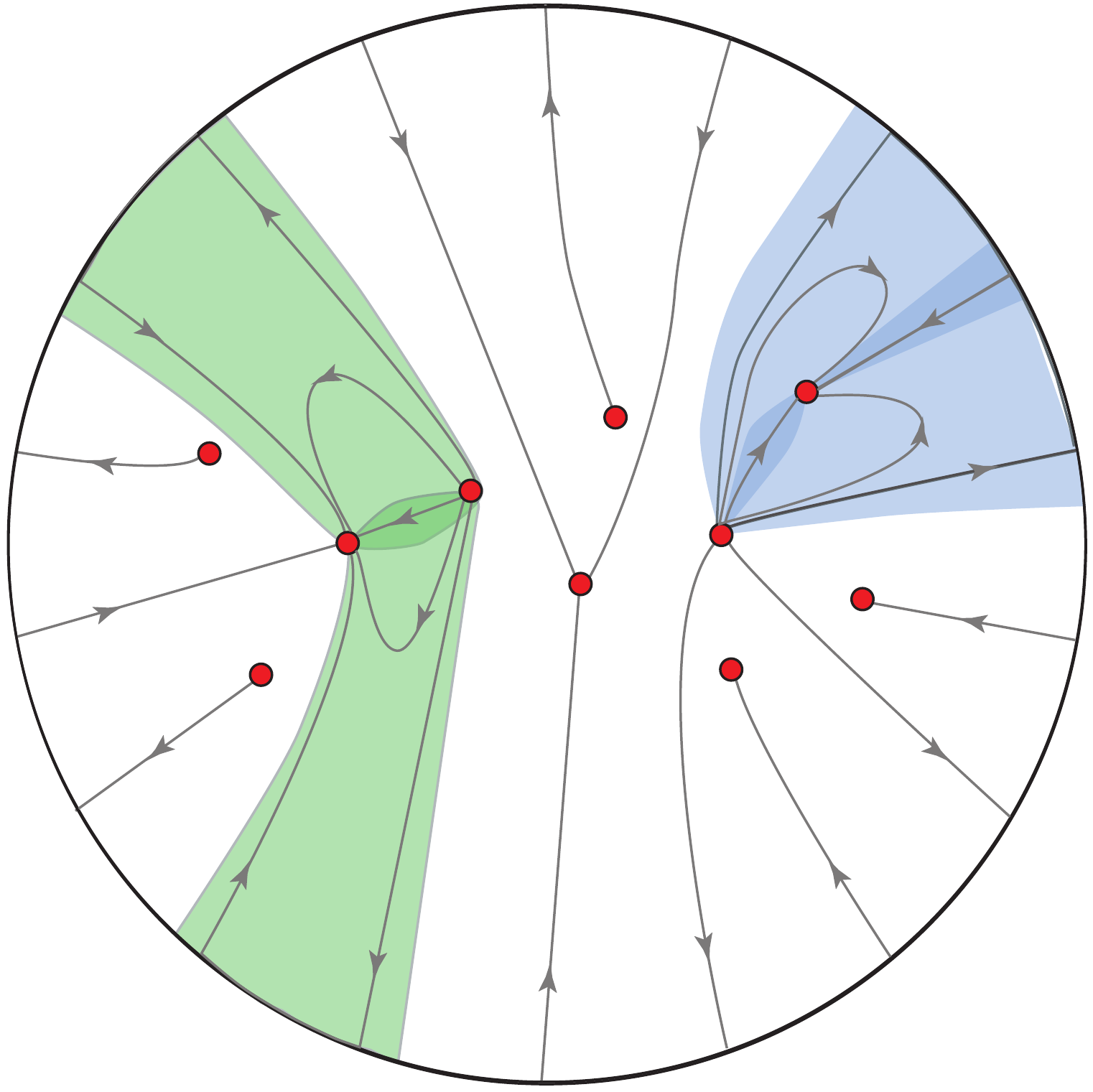}} \caption{Two connected regions determined by the separatrix graph $iP_\eps(z)\frac{\partial}{\partial z}$. }\label{fig_skeleton_sectors}\end{center}\end{figure}

\subsection{The sectoral domains in parameter space}

We want to describe the orbit space of a germ $f_\eps$ and that of $g_\eps=f_{\ov{\eps}}\circ f_\eps$. Since $g_\eps$ is close to the time-one map of $P_\eps(z)\frac{\partial}{\partial z}$, it is natural, to capture the orbits, to look at a transversal direction to the flow of $P_\eps(z)\frac{\partial}{\partial z}$, and the most natural direction is the perpendicular direction. 

We consider the intersection of the regions bounded by the separatrices of $iP_\eps(z)\frac{\partial}{\partial z}$ with the disk $\D_r$. 
The easy situation is when each intersection is connected. 
 
In that case any change of coordinate to the normal form on one of these regions of the disk in the sense of \eqref{change_nf} will be unique up to post-composition with some map $v_{\ov{\eps}}^t$ for some $t\in \R$. But these connected regions will have a disconnected limit when the two fixed points merge together. Hence, in order to have good limit properties we cut these regions into two (see Figure~\ref{fig_skeleton_sectors}(b)), using a trajectory linking the two singular points. The regions can be sectorially enlarged near the singular points to provide an open cover of $\D_r\setminus\{P_\eps(z)=0\}$ (see Figure~\ref{fig_skeleton_sectors}(c)).

The construction needs to be adapted when some intersections of the regions with $\D_r$  are disconnected. This occurs for instance when an eigenvalue at a singular point has a very small real part. Then some separatrix makes wide meandering before landing at a singular point (see Figure~\ref{meandering}). In that case we need to adapt the construction by taking the boundaries of the regions given by piecewise trajectories of vector fields $e^{i\alpha} P_\eps(z)\frac{\partial}{\partial z}$ for a finite number of real values of $\alpha$ bounded away from $\pi \Z$. In practice, this is done by changing to the time coordinate $t = \int\frac{dz}{P_\eps(z)}$ of the vector field $\frac{dz}{dt}=P_\eps(z)$. The regions will be infinite strips with piecewise linear boundaries. The bonus of this construction is that it can be extended for all non DES-generic parameter values as long as the fixed points are simple. Then we will be able to perform the construction everywhere on the complement of the discriminant set, i.e. on a region of complex codimension 1. 

\begin{definition}\label{def:sectoral_domain} A \emph{sectoral domain} is a simply connected domain in parameter space, which is an enlargement of a DES-stratum of the vector field $iP_\eps(z)\frac{\partial}{\partial z}$, on which it is possible to construct $2k$ sectors depending continuously on the parameter. \end{definition}
\begin{figure} \begin{center} \includegraphics[width=5cm]{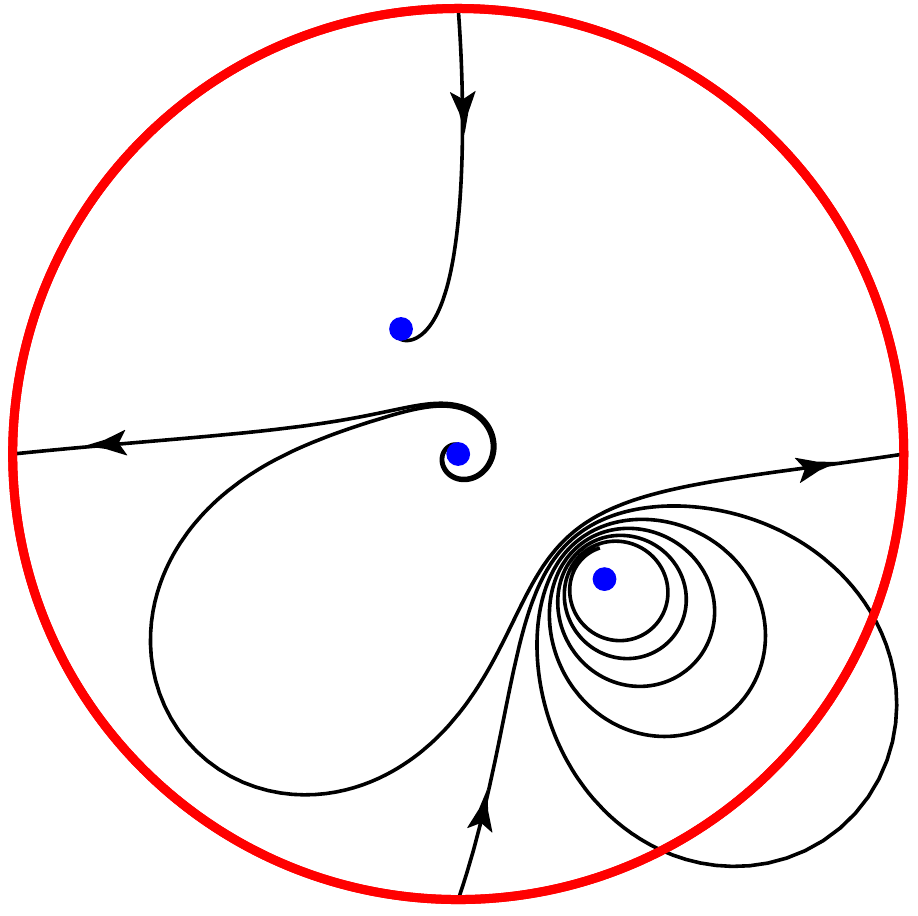}\caption{A separatrix of a polynomial vector field making wide meandering before landing at a singular point and cutting the disk into parts.}\label{meandering}\end{center} \end{figure}

\subsection{Sectors and translation domains}

\begin{definition} Let $F_{j,\eps}:= Z_{-j,\ov{\eps}}\circ f_\eps\circ Z_{j,\eps}^{-1}$ (resp. $G_{j,\eps}:= Z_{j,\eps}\circ g_\eps\circ Z_{j,\eps}^{-1}$) be the lifts of $f_\eps$ (resp. $g_\eps$) in the charts in time coordinate.
\end{definition}

Let  $\Omega_s$ be a sectoral domain. We denote by $S_{j,\eps,s}$, $j=0,\pm 1, \dots, \pm k$, where indices are $({\rm mod}\: 2k)$, the $2k$ sectors associated to $\Omega_s$ to be constructed. They are inverse images of translation domains $U_{j,\eps,s}$, $j=0,\pm 1, \dots, \pm k$ in the time coordinate, which are defined as follows.
We first consider the particular values of $\Omega_s$, for which all singular points of $iP_\eps(z)$ are nodes. 
For these values the holes in time space are all horizontal. Let $\eps\in \Omega_s$. It is known that $G_{j,\eps}$ is close to the translation by $1$, $T_1$ (see for instance Proposition 4.1 of \cite{R15}).
Let us take any vertical line $\ell_\eps$ to the left or right of principal hole in the chart $Z_{j,\eps}$ such that
\begin{enumerate}
\item there are no other holes between  $\ell_\eps$ and the principal hole,
\item the strip $B_{\ell_\eps}$ bounded by $\ell_\eps$ and $G_{j,\eps}(\ell_\eps)$  is included in the chart. \end{enumerate}
Then the translation domain $U_{j,\eps,s}$ associated to the chart $Z_{j,\eps}$ is the saturation of the strip $B_{\ell_\eps}$ by $G_{j,\eps}$ inside the chart. 
For the other values of $\eps\in \Omega_s$ we may take for $\ell_\eps$ any bi-infinite piecewise linear curve such that $\ell_\eps$ and $G_{j,\eps}(\ell_\eps)$ do not intersect, (1) and (2) above are satisfied, and $\ell_\eps$ depends continuously on $\eps$ (see Figure~\ref{strips}). 
\begin{figure}\begin{center}
\subfigure[]{\includegraphics[height=5cm]{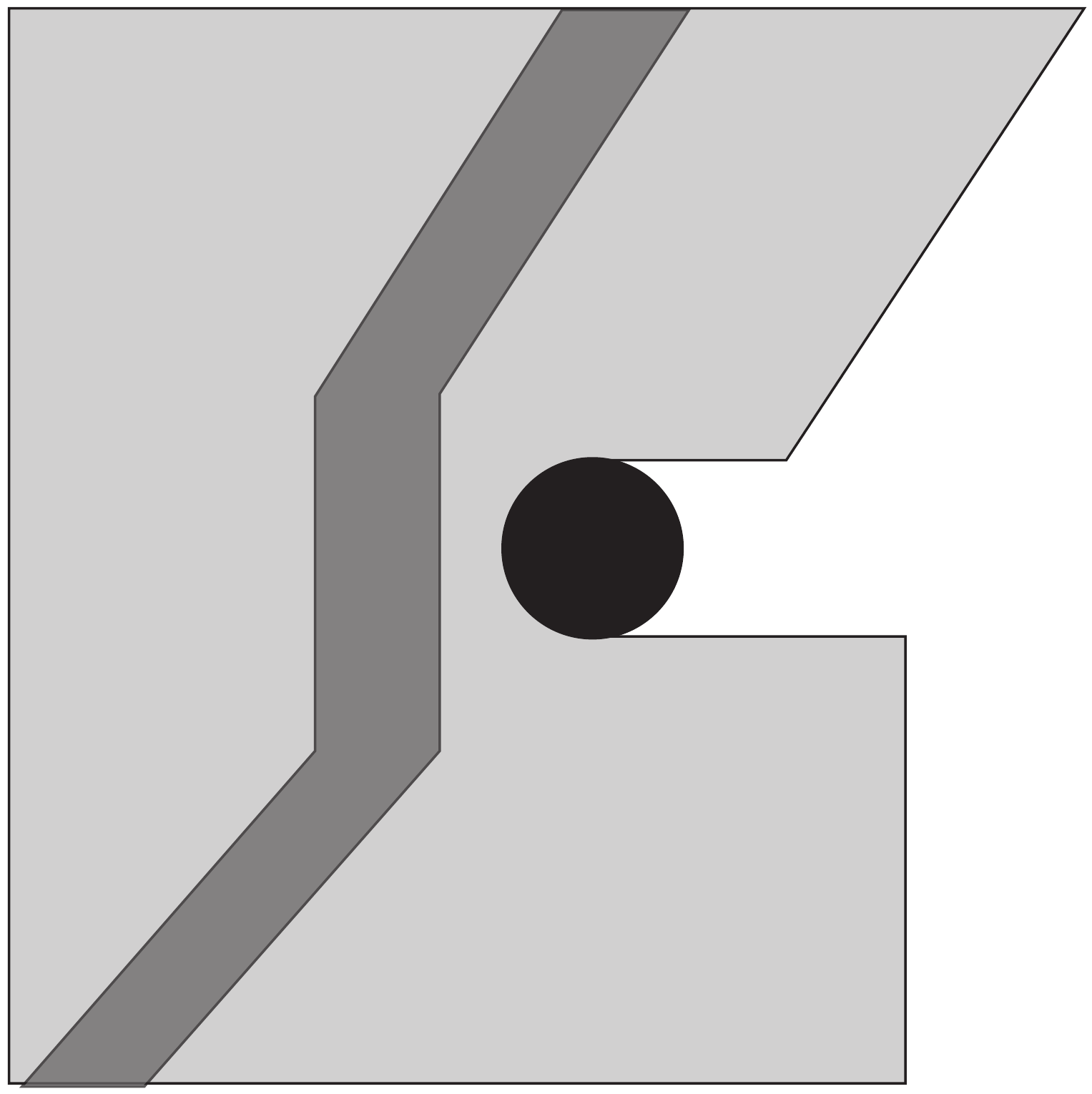}}\qquad\quad\subfigure[]{\includegraphics[height=5cm]{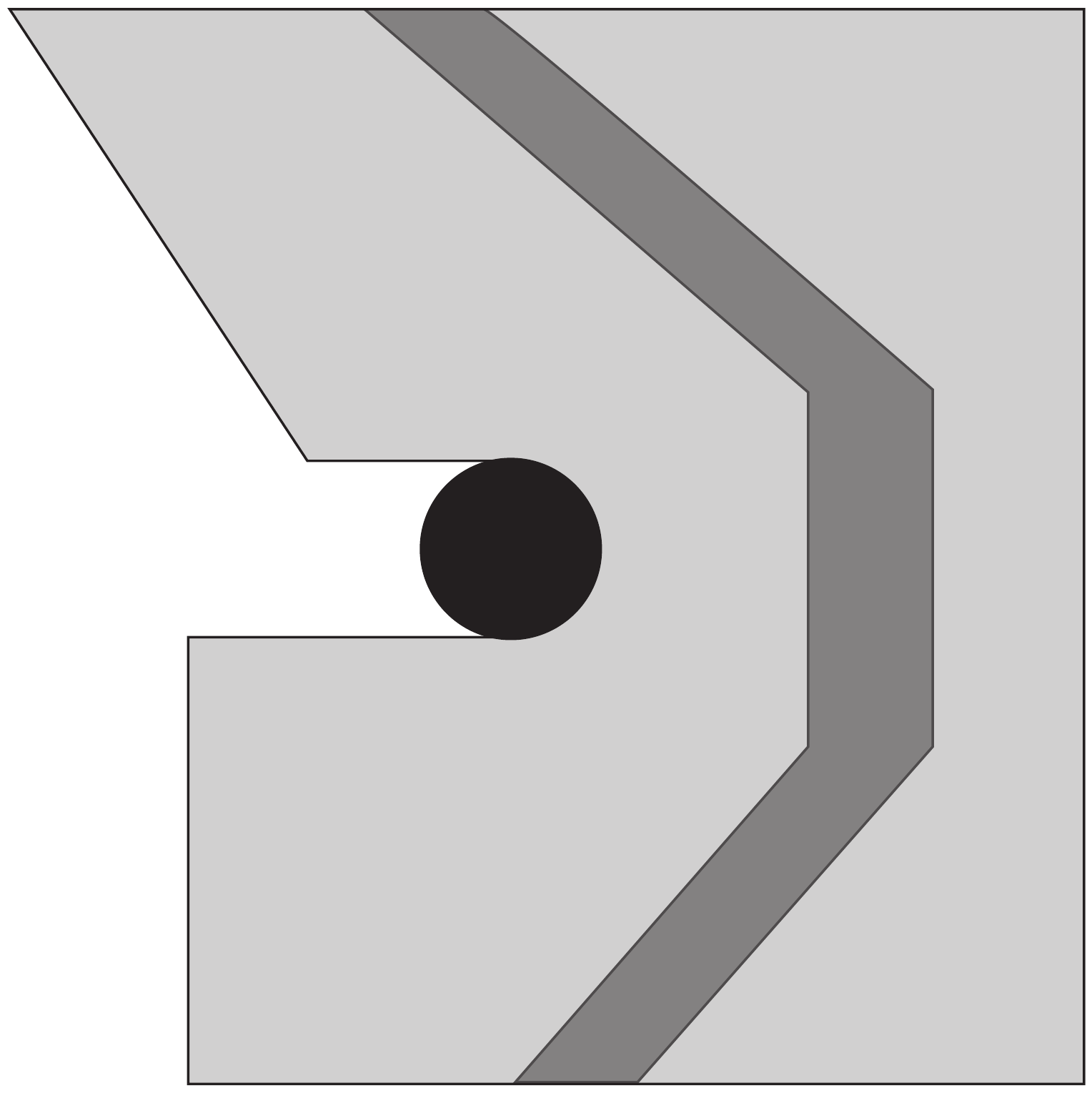}}\caption{Two strips on different sides of the fundamental hole. When there is a transition map, the slopes should be the same (bottom on the figure).}\label{strips}\end{center}\end{figure}

The sectors in $z$-space are simply $S_{j,\eps,s}=Z_{j,\eps}^{-1}(U_{j,\eps,s})$, $j=0, \pm 1, \dots, \pm k$ with indices $({\rm mod}\: 2k)$.

\subsubsection{Pairing sectoral domains} 
\begin{proposition} It is possible to cover the complement of the discriminant set in parameter space with $C(k)$ sectoral domains. The size of sectoral domains can be chosen so that the image of a sectoral domain under $\eps\mapsto \ov{\eps}$ is again a sectoral domain. Then sectoral domains can be either \begin{itemize}
\item invariant under $\eps\mapsto \ov{\eps}$,
\item or grouped by symmetric pairs. 
\end{itemize}\end{proposition}
\begin{proof} The proof can be found in \cite{R15}. The last property comes from the fact that the coefficients of $P_\eps(z)$ are real for real $\eps$. \end{proof}

If $\Omega_s$ is a sectoral domain, then we denote by $\Omega_{\ov{s}}:=\ov{\Omega_s}$ its symmetric image. This yields an involution on the set of indices, which we denote by $s\mapsto \ov{s}$.

\subsection{The Fatou coordinates}
\begin{proposition}[Definition of Fatou Coordinates]\label{prop:fatou}
 Let $f_\eps$ be a prepared germ of  type \eqref{family_prepared}. Let $F_{j,\eps}$ be the lift of $f_\eps$ in the time
  coordinate $Z_{j,\eps}$. Then for all sectoral domains $\Omega_s$, if  
      $$
        Q_{j,s} = \bigcup_{\eps\in \Omega_s\cup\{0\}} \{\eps\}\times U_{j,\eps,s},
        $$
  $j=0, \pm 1, \dots, \pm k$,    then   there exists families 
      $\{\Phi_{j,\eps,s}\}_{\eps\in\Omega_s\cup\{0\}}$       of Fatou coordinates of $f_\eps$ defined on $Q_{j,s}$ such that \begin{itemize}
  \item  
    \begin{equation}\label{eq:F STt}
    \Phi_{-j,\ov{\eps}, \ov{s}}\circ F_{j,\eps}\circ (\Phi_{j,\eps,s})^{-1}      = \STt;    \end{equation}
  \item $ \Phi_{j,\eps, s}$ is holomorphic on ${\rm int}(Q^{j,s})$ with continuous limit at $\eps=0$ independent of $s$, i.e.~
    $$
      \lim_{\eps\to 0\atop \eps\in \Omega_s} 
      \Phi_{j,\eps,s} = \Phi_{j,0},
    $$
    where the convergence is uniform on compact sets
    and $\Phi_{j,0}$ is a Fatou coordinate of $f_0$
    on $U_{j,0}$; 
    \item The families are uniquely determined
    by
    \begin{equation}\label{eq:determine Phi}
      \ov{\Phi_{-j,\ov{\eps},\ov{s}}(X_{-j,\ov{\eps},\ov{s}})} + \Phi_{j,\eps,s}(X_{j,\eps,s}) 
        = C_{j,\eps,s},
    \end{equation}
    where $X_{j,\eps,s}\in U_{j,\eps,s}$ and $X_{-j,\ov{\eps},\ov{s}}$ are base points, $\sigma\circ C_{j,\eps,s}=C_{-j,\ov{\eps},\ov{s}}$, and both $X_{j,\eps,s}$ and $C_{j,\eps,s}$ are holomorphic in $\eps\in \Omega_s$ with continuous limit at $\eps=0$.
\end{itemize}
\end{proposition}
\begin{proof} 
We take $\widetilde{\Phi}_{j,\eps,s}$ a Fatou coordinate for $G_{j,\eps}$ satisfying $\widetilde{\Phi}_{j,\eps,s}\circ G_{j,\eps} = T_1\circ \widetilde{\Phi}_{j,\eps,s}$ and depending analytically on $\eps$ with continuous limit at $\eps=0$. These are known to exist (see \cite{R15}). One way to achieve the required dependence on $\eps$ is to take a base point $X_{j,\eps,s}$ depending analytically on $\eps$ with continuous limit at $\eps=0$ independent of $s$ (a base point constant in $\eps$ and $s$ would work) and to ask that  $\widetilde{\Phi}_{j,\eps,s}(X_{j,\eps,s})=0$. 

Let $\widetilde{K}_{j,\eps,s}= \widetilde{\Phi}_{-j,\ov{\eps},\ov{s}}\circ F_{j,\eps}\circ (\widetilde{\Phi}_{j,\eps,s})^{-1}$. Then $\widetilde{K}_{j,\eps,s}$ is a diffeomorphism, which commutes with $T_1$. Quotienting by $T_1$, yields that $\widetilde{K}_{j,\eps,s}= \Sigma\circ T_{A_{j,\eps,s}}$. Moreover $\widetilde{K}_{-j,\epsbar,\ov{s}}\circ \widetilde{K}_{j,\eps,s} =T_1$, which yields $A_{j,\eps,s}+\ov{A_{-j,\ov{\eps},\ov{s}}}=1$. The result follows by letting $\Phi_{j,\eps,s} = T_{-\frac{\ov{A_{-j,\ov{\eps},\ov{s}}}}2}\circ \widetilde{\Phi}_{j,\eps,s}$ and $\Phi_{-j,\epsbar,\ov{s}} = T_{-\frac{\ov{A_{j,\eps,s}}}2}\circ \widetilde{\Phi}_{-j,\epsbar,\ov{s}}$ (details as in \cite{GR22}).

Moreover, other Fatou coordinates satisfying \eqref{eq:F STt} must have the form $T_{B_{j,\eps,s}}\circ\Phi_{j,\eps,s}$ with $B_{j,\eps,s}=\ov{B_{-j,\ov{\eps},\ov{s}}}$.
This changes $ C_{j,\eps,s}:=  \ov{\Phi_{-j,\ov{\eps},\ov{s}}(X_{-j,\ov{\eps},\ov{s}})} + \Phi_{j,\eps,s}(X_{j,\eps,s})$ to  $C_{j,\eps,s}+2 B_{j,\eps,s}$.
 \end{proof}
 
\subsection{Defining the modulus}

\begin{definition}\label{def:transition_functions} Let $f_\eps$ be a prepared germ of  type \eqref{family_prepared}, let $\Omega_s$ be a sectoral domain, and let  $\{\Phi_{j,\eps,s}\}_{\eps\in\Omega_s\cup\{0\}}$,   $j=0, \pm 1, \dots, \pm k$ be associated Fatou coordinates. The $2k$ \emph{associated transition functions} are the functions (see Figure~\ref{fleur4})
\begin{equation}\label{transition_functions}
\Psi_{\ell,\eps,s} = \begin{cases}\Phi_{\ell,\eps,s}\circ T_{-\sgn(\ell)\frac{i\pi b(\eps)}{k}}\circ (\Phi_{\ell-1, \eps, s})^{-1}, &\ell \:\text{odd},\\
\Phi_{\ell-1,\eps,s}\circ T_{\sgn(\ell)\frac{i\pi b(\eps)}{k}} \circ (\Phi_{\ell, \eps, s})^{-1}, &\ell \:\text{even},\end{cases}\end{equation}
$\ell=\pm1, \dots, \pm k$. \end{definition}
\begin{figure}\begin{center} \includegraphics[width=5cm]{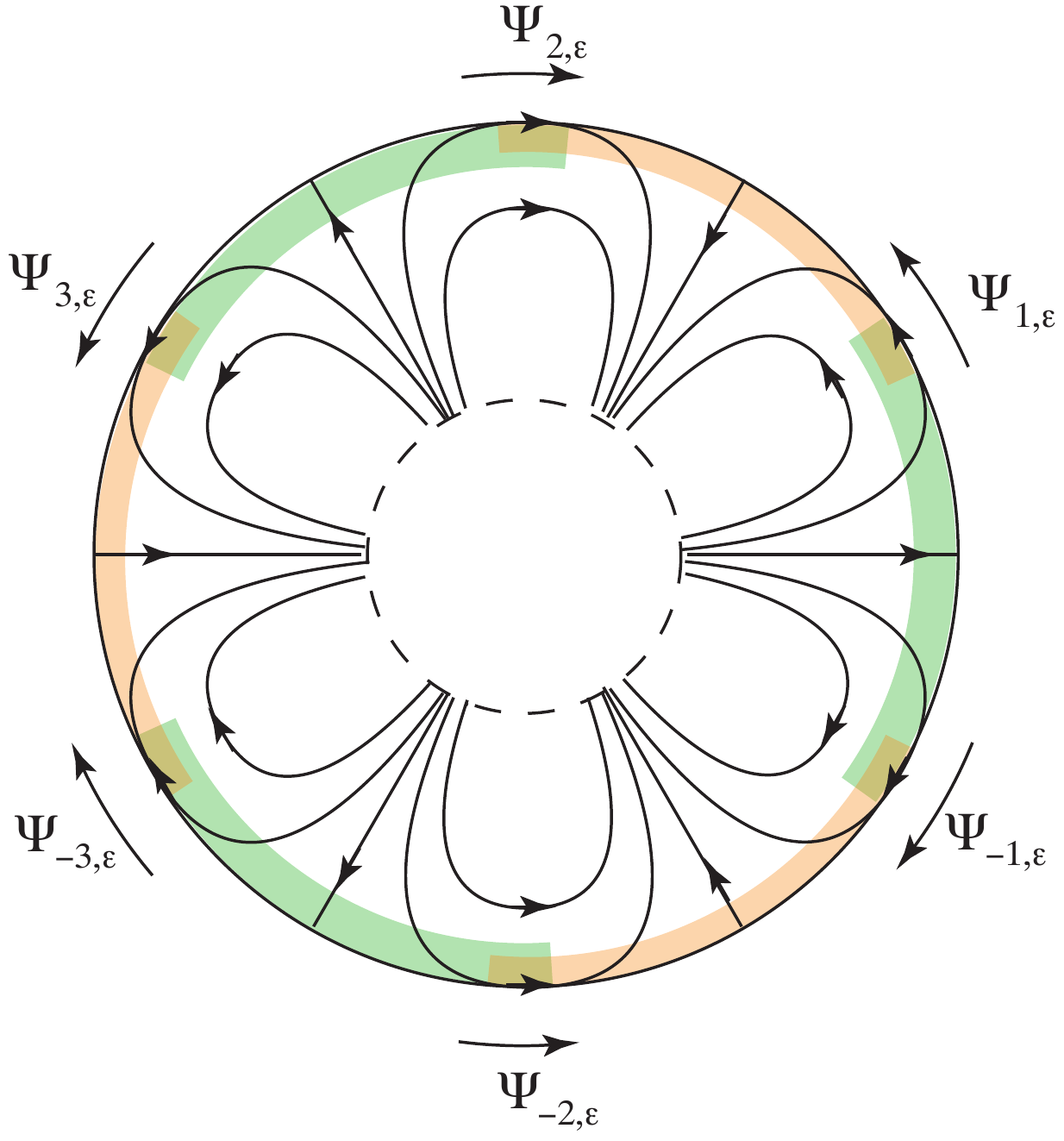}\caption{The transition functions.}\label{fleur4}\end{center}\end{figure}

\begin{proposition}
Let $f_\eps$ be a prepared germ of  type \eqref{family_prepared}, let $\Omega_s$ be a sectoral domain, and let  $\{\Psi_{\ell,\eps,s}\}_{\eps\in\Omega_s\cup\{0\}}$,   $\ell=\pm 1, \dots, \pm k$, be  associated transition functions. Then
\begin{enumerate}
\item $T_1\circ \Psi_{\ell,\eps,s}= \Psi_{\ell,\eps,s}\circ T_1$;
\item \begin{equation}\STt \circ \Psi_{\ell,\eps,s}=\Psi_{-\ell,\ov{\eps},\ov{s}}\circ \STt.
\label{eq_trans}\end{equation}
In particular, all transition functions are determined by the ones for $\ell>0$.
\item It is possible to choose Fatou coordinates so that the constant terms in the Fourier expansion of $\{\Psi_{\ell,\eps,s}\}_{\eps\in\Omega_s\cup\{0\}}$ are given by 
\begin{equation}\label{c_0} c_{\ell,\eps,s}=\sgn(\ell)(-1)^\ell\;\frac{i\pi b(\eps)}{k}.\end{equation} Such Fatou coordinates are called \emph{normalized} and the corresponding transition functions are also called \emph{normalized}.
\item If $\{\widetilde{\Psi}_{\ell,\eps,s}\}_{\eps\in\Omega_s\cup\{0\}}$,   $\ell=\pm 1, \dots, \pm k$, are other transition functions associated to other normalized Fatou coordinates, then 
there exist $B_{\eps,s}$ satisfying $B_{\eps,s}=\ov{B_{\ov{\eps},\ov{s}}}$ analytic in $\eps\in \Omega_s$ with continuous limit at $\eps=0$ such that 
\begin{equation}\label{equiv_psi} \widetilde{\Psi}_{\ell,\eps,s}= T_{-B_{\eps,s}}\circ \Psi_{\ell,\eps,s} \circ T_{B_{\eps,s}}.\end{equation} 
We say that the collections of normalized transition functions $\{\Psi_{1,\eps,s}, \dots,\allowbreak \Psi_{k,\eps,s}\}_{\eps\in\Omega_s\cup\{0\}}$ and $\{\widetilde{\Psi}_{1,\eps,s}, \dots, \widetilde{\Psi}_{k,\eps,s}\}_{\eps\in\Omega_s\cup\{0\}}$ are \emph{equivalent} and we write \begin{equation}\label{def:equiv}\{\Psi_{1,\eps,s}, \dots, \Psi_{k,\eps,s}\}_{\eps\in\Omega_s\cup\{0\}}\equiv \{\widetilde{\Psi}_{1,\eps,s}, \dots, \widetilde{\Psi}_{k,\eps,s}\}_{\eps\in\Omega_s\cup\{0\}}.\end{equation}
\item When $k$ is even, if $f_\eps$ is in prepared form and $L_{-1}(z)=-z$, then $\hat{f}_{\tilde{\eps}} =L_{-1}
 \circ f_\eps\circ L_{-1} $ is also in prepared form for the canonical parameter $\hat{\eps}$ defined in \eqref{canonical_odd}. Let $\Omega_{\hat{s}}$ be the image of $\Omega_s$ under the map $\eps\mapsto \hat{\eps}$. 
 If $\{\Psi_{\ell,\eps,s}\}_{\eps\in\Omega_s\cup\{0\}, \ell=1, \dots, k}$ are normalized transition functions for $f_\eps$ and $$\widehat{\Psi}_{\ell, \hat{\eps}, \hat{s}}= \Sigma\circ T_{\frac12}\circ \Psi_{k+1-\ell,\eps,s}\circ \Sigma\circ T_{-\frac12}$$  then $\{\widehat{\Psi}_{1, \hat{\eps}, \hat{s}}, \dots, \widehat{\Psi}_{k, \hat{\eps}, \hat{s}}\}_{\hat{\eps}\in\Omega_{\hat{s}}\cup\{0\}}$ are normalized transition functions for $\hat{f}_{\hat{\eps}}$.
 We write 
 \begin{equation}\left( \{\Psi_{1,\eps,s}, \dots,\Psi_{1,\eps,s} \}_{\eps\in\Omega_s\cup\{0\}}\right)\cong \left(\{\widehat{\Psi}_{1, \hat{\eps}, \hat{s}}, \dots, \widehat{\Psi}_{k, \hat{\eps}, \hat{s}}\}_{\hat{\eps}\in\Omega_{\hat{s}}\cup\{0\}}\right)\label{equiv2}\end{equation}
 \end{enumerate}
\end{proposition} 

\begin{remark}\label{rem:constants}
Note that the constant terms $c_{\ell,\eps,s}$ in \eqref{c_0} coincide precisely with the change of time coordinates $Z_{j,\eps}$ between the corresponding sectors in \eqref{relation:Z_j}.
\end{remark}

\begin{definition} Let $f_\eps$ be a prepared germ of  type \eqref{family_prepared}. \begin{enumerate} 
\item For $k$ odd, the modulus of $f_\eps$ is given by the $(kC(k)+3)$-tuple
\begin{equation} \mathcal{M}(f_\eps)=\left(k,\eps,b_\eps, \left(\{\Psi_{1,\eps,s}, \dots,\Psi_{k,\eps,s}\}_{\eps\in\Omega_s\cup\{0\}}\right)_s/\equiv \right), \label{modulus}\end{equation}
where $\{\Psi_{\ell,\eps,s}\}_{\eps\in\Omega_s\cup\{0\}}$   are the associated normalized transition functions to a sectoral domain $\Omega_s$. 
This is also the modulus of $f_\eps$ for $k$ even under conjugacy tangent to the identity.
\item For $k$ even, the modulus of $f_\eps$ is given by the quotient of $\mathcal{M}(f_\eps)$ by $\cong$: 
\begin{equation} \mathcal{N}(f_\eps)=\left(k,\eps,b_\eps, \left(\{\Psi_{1,\eps,s}, \dots,\Psi_{k,\eps,s}\}_{\eps\in\Omega_s\cup\{0\}}\right)_s/\equiv \right)/\cong, \label{modulus_even}\end{equation}
where 
\begin{align*}\begin{split}&\left(k,\eps,b_\eps, \left(\{\Psi_{1,\eps,s}, \dots,\Psi_{k,\eps,s}\}_{\eps\in\Omega_s\cup\{0\}}\right)_s/\equiv \right) \\&\qquad\qquad \cong \left(k,\hat{\eps},b_{\hat{\eps}}, \left(\{\widehat{\Psi}_{1,\hat{\eps},\hat{s}}, \dots,\widehat{\Psi}_{k,\hat{\eps},\hat{s}}\}_{\hat{\eps}\in\Omega_{\hat{s}}\cup\{0\}}\right)_{\hat{s}}/\equiv \right).\end{split}\end{align*}
\end{enumerate}
\end{definition}

\subsection{The classification theorem}

\begin{theorem}\label{class_thm}
Two prepared unfoldings of antiholomorphic parabolic germs  of  type \eqref{family_prepared} are analytically conjugate if and only if they have the same modulus. 
\end{theorem}
\begin{proof} If two families are analytically conjugate, then they obviously have the same modulus. 
Conversely, suppose that two prepared families $f_\eps$ and $\tilde{f}_{\tilde{\eps}}$ have the same modulus. In the case $k$ odd, then $\eps=\tilde{\eps}$ by Corollary~\ref{cor:canonical}, and it is of course possible to suppose that their normalized transition functions are equal:
$\Psi_{\ell,\eps,s}= \widetilde{\Psi}_{\ell,\eps,s}$. 
When $k$ is even, the same is true, possibly after conjugating $\tilde{f}_{\tilde{\eps}}$ by $L_{-1}$, in which case the new canonical parameter becomes $\hat{\tilde{\eps}}= \eps$. 

Moreover, the Fatou coordinates have been chosen so that $\Psi_{\ell, 0,s}$ are independent of $s$. 
For $\eps\in \Omega_s$, a conjugacy is defined by 
$$H_{\eps,s}(z)= Z_{j,\eps}^{-1} \circ (\widetilde{\Phi}_{j,\eps,s})^{-1}\circ \Phi_{j,\eps,s}\circ Z_{j,\eps}, \qquad j=0, \pm1, \dots, \pm k,$$
where $\Phi_{j,\eps,s}$ and $\widetilde{\Phi}_{j,\eps,s}$ are the normalized Fatou coordinates of $f_\eps$ and $\tilde{f}_\eps$ respectively. 
We claim that $H_{\eps,s}$ is well defined over $\D_r$. Since the conjugacy we are constructing is also  a conjugacy between $g_\eps=f_{\ov{\eps}}\circ f_\eps$ and $\tilde{g}_\eps=\tilde{f}_{\ov{\eps}}\circ \tilde{f}_\eps$, and since full details have been given for the latter case in  \cite{R15}, we explain the ideas and skip some details. The intersection of two sectors has connected components of two forms (see Figure~\ref{ssecteurs}):
\begin{itemize}
\item subsectors from one fixed point of $g_\eps$ to the boundary: on such a subsector the result follows from \eqref{eq_trans}.
\item subsectors joining two singular points, sometimes called \emph{gate sectors} (the name comes from \cite{O}). The transition map between Fatou coordinates over a gate sector is a translation. The normalization of a transition map is such that this translation depends only on the normal form. Indeed, crossing a gate sector like along the blue line in Figure~\ref{skeleton_time} is the same as  turning aroung the singular points on one side of the blue line or on  the other side (of course in the appropriate direction) and taking into account the changes of time \eqref{relation:Z_j} from one sector to the next. And the period of a singular point $z_n$ is $\frac{2\pi i}{g_\eps'(z_n)}= \frac{2\pi i}{\tilde{g}_\eps'(z_n)}$. Hence the translation given by the transition over of a gate sector is the same for $f_\eps$ and for $\tilde{f}_\eps$.
\end{itemize}
\begin{figure}\begin{center}
\includegraphics[width=6cm]{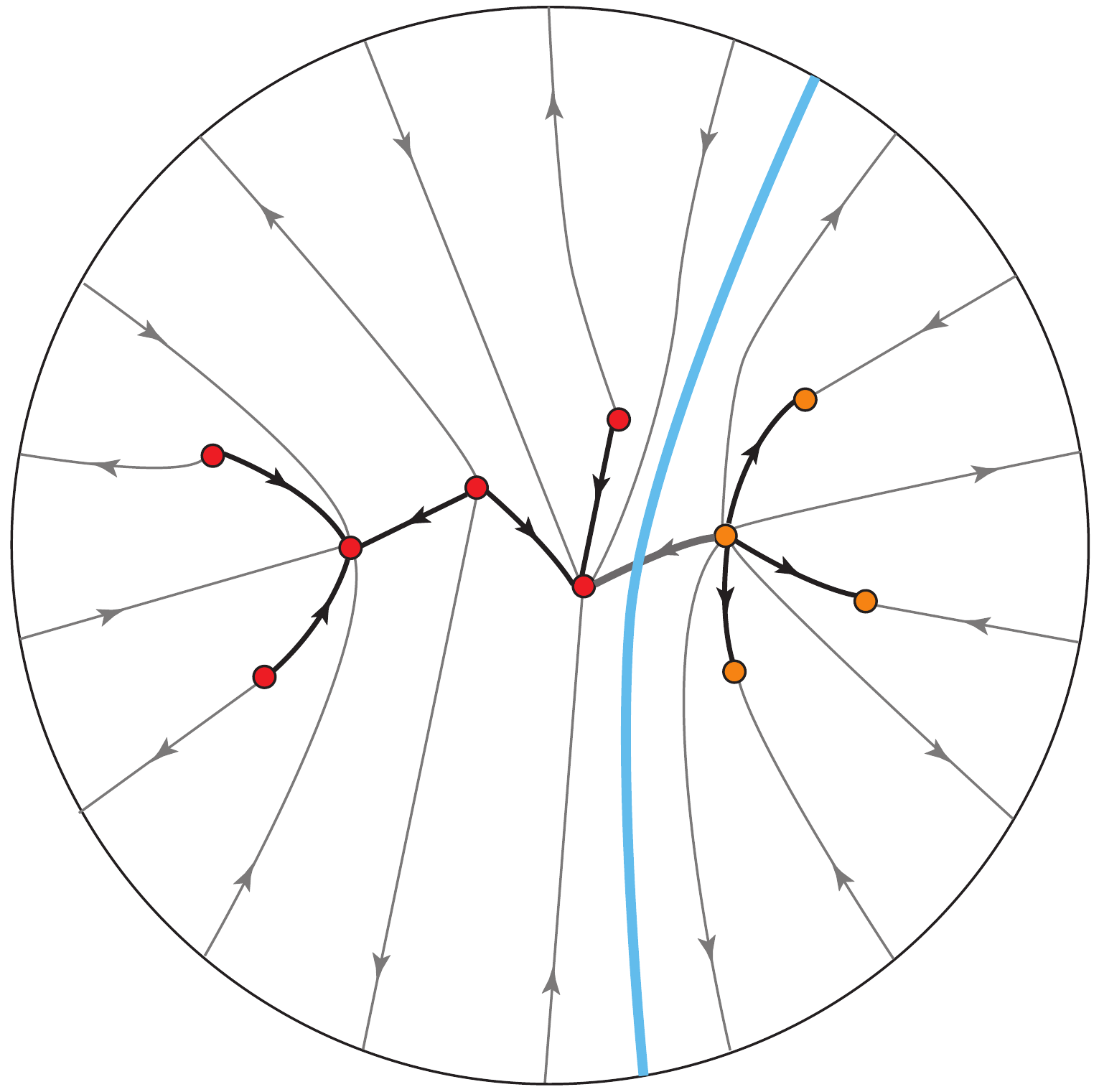}\caption{The change of time of the crossing of a gate sector (in gray) from top to bottom along the blue line is the same as the change of time when turning around the singular points on the left in the positive direction, or turning around the singular points on the right in the negative direction and, in both cases, taking also into account the changes of time \eqref{relation:Z_j} from one sector to the next.}\label{skeleton_time}\end{center}\end{figure}

Now, suppose that $\Omega_s\cap \Omega_{s'}\neq \emptyset$. Then $H_{\eps,s'}^{-1} \circ H_{\eps,s}$ commutes with $g_\eps$, and is equal to the identity for $\eps=0$.  If the modulus is non trivial (i.e. not all transition functions are identically translations), then $H_{\eps,s'}^{-1} \circ H_{\eps,s} = g^{\circ \frac{m}{n}}$ for some nonzero $n$ independent of $\eps$ by Proposition~\ref{prop:symmetries} below.  Since $H_{0,s'}^{-1}\circ H_{0,s}={\rm id}$ because the $\Psi_{\ell, 0,s}$ are independent of $s$, then $m=0$ and the $H_{\eps,s}$ are analytic extensions of each other when $s$ varies and yield a uniform bounded conjugacy  $H_\eps$ outside the parameter values in the discriminant set. Hence the conjugacy can be analytically extended to the discriminant set. 

If the modulus is trivial, then the $H_{\eps,s}$ need to be corrected before being glued in a uniform way. Indeed $H_{\eps,s'}^{-1} \circ H_{\eps,s} = g^{\circ t(\eps)}$ for some real $t(\eps)$ which has the property that $t(0)=0$. We want to modify the normalized Fatou coordinates so as to force that $t(\eps)=0$. This is done by choosing normalized Fatou coordinates with one fixed base point, for instance $z=r$ (resp. $z=r'$) for $\Phi_{0,\eps,s}$ (resp. $\widetilde{\Phi}_{0,\eps,s}$). Then $g^{\circ t(\eps)}(r)=r$, which yields $t(\eps)=0$ since $t$ is continuous and $t(0)=0$. \end{proof} 

The following proposition is well known (see for instance \cite{R15}).

\begin{proposition}\label{prop:symmetries} 
Let $g_\eps$ be an unfolding of a holomorphic parabolic germ. 
Then 
\begin{enumerate}
\item either $g_\eps$ is conjugate to the normal form $v_\eps^1$ and any holomorphic family of diffeomorphisms $h_\eps$ commuting  with $g_\eps$ has the form $h_\eps= g_\eps^{\circ \alpha(\eps)}$ for $\alpha(\eps)$ analytic;
\item or there exists $q\in \N*$ such that any holomorphic family of diffeomorphisms $h_\eps$ commuting  with $g_\eps$ has the form $h_\eps= g_\eps^{\circ \frac{p}{q}}$ for some $p\in \Z$. In particular if $
\lim_{\eps\to 0} h_\eps = {\rm id}$, then $h_\eps\equiv {\rm id}$. 
\end{enumerate} 
\end{proposition} 
\begin{proof} In each Fatou coordinate of $g_\eps$, then $h_\eps$ commutes with $T_1$, i.e. is of the form $T_{\alpha(\eps)}$. For $h_\eps$ to be uniformly defined over $\D_r$, then $T_{\alpha(\eps)}$ must commute with the transition functions. In case (1), the transition functions are translations and any translation commutes with them. In case (2), there is a maximum $q\in \N$ such $T_{\frac1q}$ commutes with the transition functions. Then $\alpha(\eps)= \frac{p}{q}$ is constant in $\eps$.  \end{proof}

\begin{corollary}
Two prepared families  of  type \eqref{family_prepared} are analytically conjugate under a conjugacy tangent to the identity if and only if their second iterates $g_\eps$ and $\tilde{g}_\eps$ are analytically conjugate under a conjugacy tangent to the identity.
\end{corollary}
\begin{proof} One direction is obvious. For the other direction, it is important to use that $g_\eps$ and $\tilde{g}_\eps$ have representatives of the modulus satisfying \eqref{eq_trans}. Then an equivalence between them constructed as in the proof of Theorem~\ref{class_thm} (and hence tangent to the identity) yields an equivalence between $f_\eps$ and $\tilde{f}_\eps$. \end{proof}

\begin{corollary}
A prepared family  of  type \eqref{family_prepared} is analytically conjugate to its normal form $\sigma_\circ v_\eps^{\frac12}$ if and only if all the transition maps $\Psi_{j,\eps,s}$ are translations. \end{corollary}

\section{Antiholomorphic parabolic unfolding with an invariant real analytic curve}

\subsection{The case $\eps=0$} 

This case has been studied in \cite{GR21}. Suppose that an antiholomorphic parabolic germ $f_0$ keeps invariant a germ of real analytic curve. This property is invariant under holomorphic conjugacy and can be read on the modulus. 
Indeed, modulo a conjugacy, we can suppose that $f_0$ preserves the real axis, and hence commutes with $\sigma$. This in turn implies that the transition maps satisfy
\begin{equation}\Sigma\circ \Psi_\ell= \Psi_{-\ell}\circ \Sigma.\label{cond:sym0}\end{equation}
Together with \eqref{eq_trans}, this yields that for all $\ell$
\begin{equation} T_{\frac12} \circ \Psi_\ell= \Psi_\ell\circ T_{\frac12},\label{cond_square-root0}\end{equation} which is precisely the condition for $g_0=f_0\circ f_0$ to have a holomorphic square root (see for instance \cite{I}). Indeed, this is natural since \eqref{cond:sym0} yields that $f_0$ commutes with $\sigma$ and then that $\sigma\circ f_0$ is a holomorphic square root of $g_0$. 

The converse is also true.

\begin{theorem} \cite{GR21} Let $f_0$ be antiholomorphic parabolic germ. We have the equivalences:
\begin{enumerate} 
\item $f_0$ keeps invariant a germ of real analytic curve.
\item $f_0$ is analytically conjugate to a germ with real coefficients.
\item The modulus of $f_0$ satisfies \eqref{cond:sym0}.
\item The modulus of $f_0$ satisfies \eqref{cond_square-root0}.\end{enumerate}\end{theorem}

\subsection{The unfolding} 
We now consider a prepared generic unfolding $f_\eps$ of $f_0$. If we limit ourselves to real values of $\eps$, then it makes sense to have $f_\eps$ preserving a germ of real analytic curve, which is tangent to the real axis since $f_\eps$ is prepared. If $z= x+iy$, this germ of real analytic curve has the form $y= \alpha(x, \eps)= O(P_\eps(x))$, since the fixed points are real for real $\eps$ and belong to the invariant curve. This yields a local holomorphic diffeomorphism $z\mapsto \beta_\eps(z)=z+ i\alpha(z,\eps)$, which preserves the prepared character. Let us now consider 
$\tilde{f}_\eps=\beta_\eps^{-1} \circ f_\eps\circ \beta_\eps$. Then for real $\eps$, $\tilde{f}_\eps$ sends a neighborhood of $0$ on the real axis to the real axis.
For complex $\eps$, this yields  
$\tilde{f}_{\epsbar}(\ov{z})= \ov{\tilde{f}_\eps(z)},$ which in turn yields that 
$\tilde{g}_\eps(z):=\tilde{f}_{\epsbar}\circ \tilde{f}_\eps(z)= \ov{\tilde{f}_\eps(\ov{\tilde{f}_\eps(z)})}= (\sigma_\circ \tilde{f}_\eps)\circ (\sigma_\circ \tilde{f}_\eps)(z)$, i.e. $\tilde{g}_\eps$ has the  holomorphic square root $\sigma_\circ \tilde{f}_\eps$. Therefore, $g_\eps=f_{\epsbar} \circ f_\eps$ also has a holomorphic square root. 
\medskip

Hence we have the following theorem.

\begin{theorem}\label{thm:inv_curve} Let $f_\eps$ be a prepared germ of antiholomorphic parabolic unfolding. We have the equivalences:
\begin{enumerate} 
\item For real values of $\eps$, $f_\eps$ preserves a germ of real analytic curve depending real analytically on $\eps$.
\item The square $g_\eps=f_{\epsbar} \circ f_\eps$ has a holomorphic square root tangent to the identity.
\item The modulus of $f_\eps$ satisfies 
\begin{equation} T_{\frac12} \circ \Psi_{\ell,\eps,s}= \Psi_{\ell,\eps, s} \circ T_{\frac12},\label{cond_square-root}\end{equation}
\item The modulus of $f_\eps$ satisfies 
\begin{equation}\Sigma\circ \Psi_{\ell,\eps,s}= \Psi_{-\ell, \epsbar, \ov{s}}\circ \Sigma.\label{cond:sym}\end{equation}
\end{enumerate}\end{theorem}
\begin{proof} 

$(1)\Rightarrow (2)$ is shown above. 

\noindent$(2) \Rightarrow (3)$. Let $h_\eps$ be a holomorphic square root of $g_\eps$ tangent to the identity. In particular, $h_\eps$ sends (approximately) a sector $S_{j,\eps,s}$ to the same sector. Then $ \Phi_{j,\eps,s}\circ Z_{j,\eps}\circ h_\eps\circ Z_{j,\eps}^{-1} \circ \Phi_{j,\eps,s}^{-1}=T_{\frac12}$, and since $h_\eps$ is globally defined, then $T_{\frac12}$ must commute with the  $\Psi_{\ell,\eps, s}$, yielding \eqref{cond_square-root}.

\noindent$(3) \Leftrightarrow (4)$ because of \eqref{eq_trans}.

\noindent$(4) \Rightarrow (1)$. Let $$\zeta_{j,\eps,s} = Z_{-j,\epsbar}^{-1} \circ \Phi_{-j,\epsbar,\ov{s}}^{-1} \circ \Sigma \circ \Phi_{j,\eps,s}\circ Z_{j,\eps}.$$
First note that $\zeta_{j,\eps,s}$ is well defined independently  of the freedom on Fatou coordinates because of \eqref{eq:determine Phi}. Note that $\zeta_{j,\eps,s}$ is independent of $j$, yielding a well defined  $\zeta_{\eps,s}$ on $\D_r\setminus\{P_\eps(z)=0\}$ for $\eps\in\Omega_s$. This follows from \eqref{cond:sym} on the intersection sectors to the boundary. On the gate sectors, joining two singular points, it follows from the proof of Theorem~\ref{class_thm} that the translations $\mathcal{T}_{\ell,\eps,s}$ along gate sectors (crossed in symmetric directions with respect to the real axis for $(\eps,s)$ and $(\epsbar,\ov{s})$) satisfy 
$\mathcal{T}_{\ell,\eps,s}\circ \Sigma = \Sigma \circ \mathcal{T}_{-\ell,\epsbar, \ov{s}}$. 

Because $\zeta_{\eps,s}$ is bounded in the neighborhood of $P_\eps(z)=0$, it can be extended to this set. 
Moreover, $\zeta_{\eps,s}$ depends antiholomorphically on $\eps$ and $\zeta_{\epsbar,\ov{s}}\circ \zeta_{\eps,s}={\rm id}$.

Since $T_{\frac12}$ and $\Sigma$ commute, it  follows from \eqref{cond:sym} and \eqref{eq_trans} that $\zeta_{\eps,s}\circ f_\eps$ is a holomorphic square root of $g_\eps$ over $\Omega_s$, whose limit is tangent to the identity when $\eps\to 0$. 
On the intersection $\Omega_s\circ \Omega_{s'}$, $\zeta_{\eps,s}\circ f_\eps$ and $\zeta_{\eps,s'}\circ f_\eps$ are two holomorphic square roots of $g_\eps$, whose limit is tangent to the identity when $\eps\to 0$. 
By uniqueness of such square roots, we have $\zeta_{\eps,s}=\zeta_{\eps,s'}$. Hence $\zeta_\eps$ is uniformly defined outside the discriminantal set and bounded there, yielding that  it  can be extended antiholomorphically to this set.

Now, restricting to real values of $\eps$, $\zeta_\eps$ is an antiholomorphic involution depending real-analytically on $\eps$. 
Let us look at the equation of fixed points $\zeta_\eps(z)=z$. Since $\zeta_0'(0)=1$, then letting $z=x+iy$, by the implicit function theorem the equation for the imaginary parts yields $y-q(\eps,x)=0$, with $q$ real-analytic in $\eps$ and $x$. Let $V(x,y,\eps)=0$ be the equation for the real parts.
Since  $\zeta_s$ is an involution, it has no isolated fixed points. Hence $y-q(\eps,x)$ divides $V(x,y,\eps)$. Let $h_\eps(z) = z+iq(z,\eps)$. Then $\chi_\eps= h_\eps^{-1}\circ \zeta_\eps\circ h_\eps$ fixes the real axis. By the identity principle $\sigma\circ \chi_\eps={\rm id}$, yielding that $\chi_\eps=\sigma$ and that $\zeta_\eps$ is the Schwarz reflection with respect to the analytic curve $y=q(\eps,x)$. 
Let $z$ be any fixed point of $\zeta_\eps$. Since $\zeta_\eps$ and $f_\eps$ commute, then $\zeta_\eps(f_\eps(z))=f_\eps(z)$, i.e. $f_\eps(z)$ is also a fixed point of $\zeta_\eps$. Hence the curve $y=q(\eps,x)$ is invariant by $f_\eps$.
\end{proof}

\section{Antiholomorphic square root of a germ of holomorphic parabolic unfolding}

The formal normal form of a holomorphic parabolic germ is invariant under rotations of order $k$ (modulo a reparametrization), while that of an antiholomorphic germ in prepared form has the real axis as a symmetry axis. Each invariance requires a quotient in the definition of the modulus of the corresponding parabolic germ or its unfoldings. For these respective quotients, we will need to use actions of the rotation group $R_k$ of order $k$ and of the symmetry with respect to an axis  $e^{i\frac{\pi m}{k}}\R$ on the set of indices $\{\pm1, \dots, \pm k\}$ of the transition maps. 
 We start by defining these actions.

\subsection{Actions on the set of indices}

\begin{definition}\label{def:sym_indices}
\begin{itemize}
\item Let $\iota: \{\pm1, \dots, \pm k\}\rightarrow \{1, \dots 2k\}$ be defined as
$$\iota(j) = \begin{cases} j, &j>0\\
2k+1+j, &j<0.\end{cases} $$
\item The rotation group $R_k=\{r_0, r_2,\dots, r_{2(k-1)}\}$ with $r_m (w) = e^{i\frac{\pi m}{k}}w$ acts on the set of indices $\pm 1, \dots, \pm k$ as  $r_{2m}(j) = \iota^{-1} (q(\iota(j)+2m))$, where $q(s)\in \{1, \dots, 2k\}$ and $q(s)$ is congruent to $s$ (${\rm mod}\ 2k$). (By abuse of notation, $r_{2m}$ denotes both the rotation and its action on the set of indices.)
\item The symmetry $\xi_0$ with respect to $\R$ on the set of indices $\{\pm 1, \dots, \pm k\}$  is defined as $\xi_0(j)=-j$. 
\item 
The symmetry $\xi_m$ with respect to the line $e^{i\frac{\pi m}{k}}\R$  on the set of indices $\{\pm 1, \dots, \pm k\}$  is defined as 
$\xi_m = r_m\circ \xi_0\circ r_{m}^{-1}$ for $m=0, \dots, k-1$ (see Figure~\ref{fleur5}).
\end{itemize}\end{definition}

\subsection{The case $\eps=0$}

This case has been studied in \cite{GR21}. A holomorphic parabolic germ 
\begin{equation}g(z) = z + z^{k+1} + o(z^{k+1})\label{hol_parabolic}\end{equation} has $k$ formal antiholomorphic square roots of the form \begin{equation}f(\zbar) = e^{i\frac{2\pi m}{k}} \zbar +\frac12 e^{i\frac{2\pi m}{k}} \zbar^{k+1} + o(\zbar^{k+1}),\label{square-root_theta}\end{equation} $m=0, \dots, k-1$. 
Denoting $\Psi_j=\Psi_{j,0,s}$, $j=\pm1, \dots, \pm k$, defined as in Definition~\ref{def:transition_functions}, the analytic part of the modulus of $g$ is composed of the $2k$-tuple of normalized transition functions $(\Psi_{1},\dots, \Psi_{k}, \Psi_{-k}, \dots, \Psi_{-1})$ quotiented by:
\begin{itemize}\item  the action of $\C$ corresponding to conjugating all $\Psi_{j}$ by translations $T_c$;
\item the action of the rotation group $R_k$ of order $k$. The action  of $r_{2m}$ is given by: $(\Psi_{1},\dots, \Psi_k,\Psi_{-k}\dots, \Psi_{-1})\mapsto \left(\Psi_{r_{2m}(1)},\dots, \Psi_{r_{2m}(k)},\Psi_{r_{2m}(-k)}\dots, \Psi_{r_{2m}(-1)}\right)$.
\end{itemize}

 \begin{theorem} \cite{GR21} The formal square root \eqref{square-root_theta} is antiholomorphic if and only if the modulus satisfies a symmetry condition with respect to the symmetry axis $e^{i\frac{\pi m}{k}}\R$. If $\xi_m(j)$ is the symmetric index of $j$ with respect to $e^{i\frac{\pi m}{k}}\R$, then this symmetry condition takes the form $$\STt \circ \Psi_j=\Psi_{\xi_m(j)}\circ \STt
 $$ for some representative of the modulus.\end{theorem}
\begin{figure}\begin{center} \includegraphics[width=5cm]{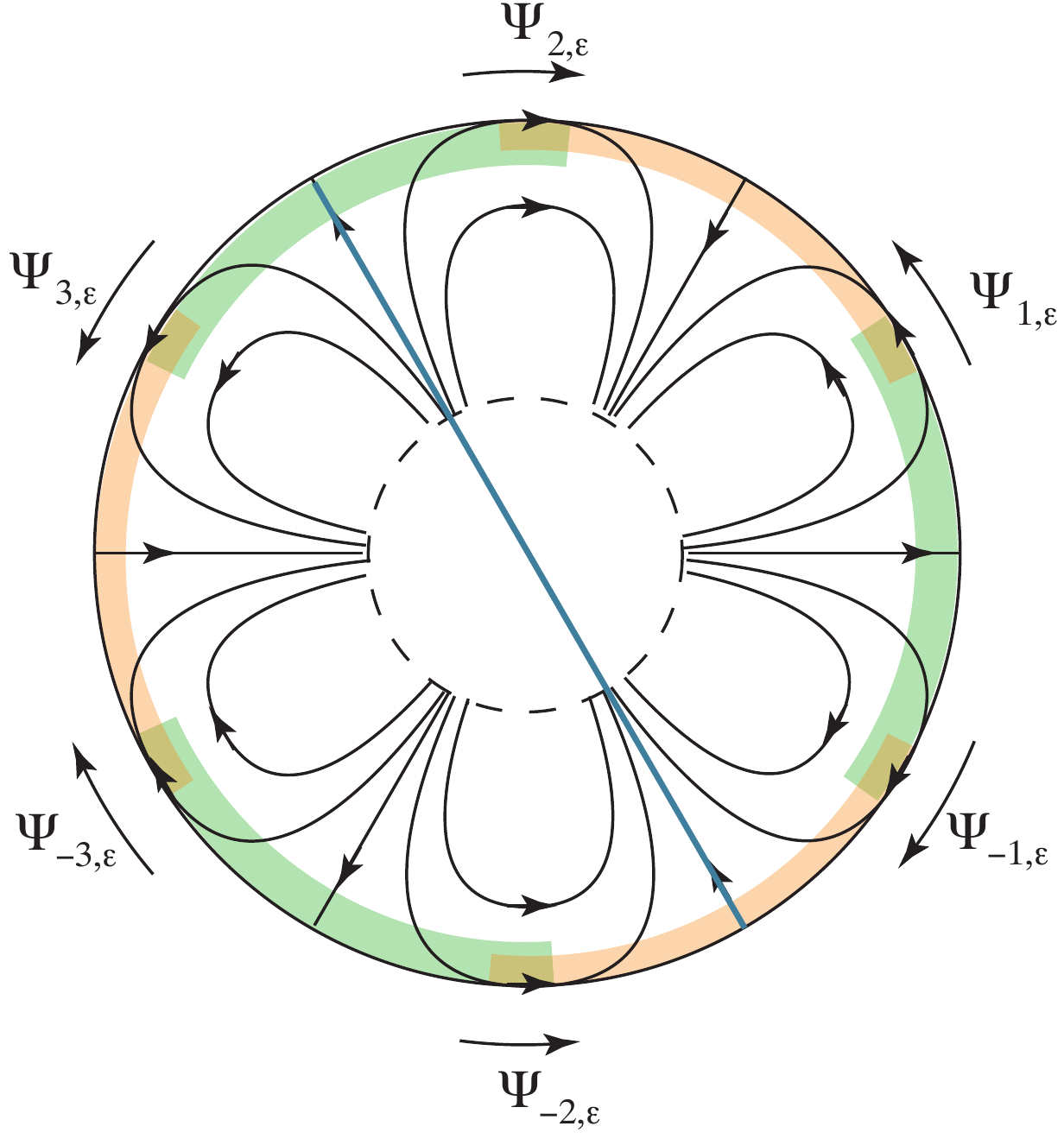}\caption{For $k=3$, the symmetry condition on the indices with respect to the symmetry axis $e^{\frac{2\pi i}{3}}\R$ is given by the involution $\xi_2(1)=-3$, $\xi_2(2)=3$, $\xi_2(-1)= -2$.}\label{fleur5}\end{center}\end{figure}

\subsection{The unfolding case}

Generic holomorphic unfoldings of a parabolic germ \eqref{hol_parabolic} have been studied in \cite{R15}. They can also be put in a prepared form with canonical parameters
\begin{equation} g_\eps(z) = z + P_\eps(z) (1 +M_\eps(z) + P_\eps(z)N_\eps(z)),\label{hol_parabolic_unf}\end{equation} where $P_\eps$ is defined in \eqref{def:P} and $M_\eps$ is a polynomial in $z$ of degree at most $k$. 

Sectoral domains can be defined as in Definition~\ref{def:sectoral_domain} and transition functions for each sectoral domain as in Definition~\ref{def:transition_functions}.

\begin{definition} Let $g_\eps$ be a prepared germ of  type \eqref{hol_parabolic_unf}. The modulus of $g_\eps$ is given by the equivalence class of $(3+2kC(k))$-tuples (see Figure~\ref{fleur4})
\begin{equation} \mathcal{M}(f_\eps)=\left(k,\eps,b_\eps, \left(\left(\{\Psi_{\pm1,\eps,s}, \dots,\Psi_{\pm k,\eps,s}\}_{\eps\in\Omega_s\cup\{0\}}\right)/\equiv\right)_s \right)/\cong, \label{modulus}\end{equation}
where $\{\Psi_{\ell,\eps,s}\}_{\eps\in\Omega_s\cup\{0\}}$  are the associated normalized transition functions to a sectoral domain $\Omega_s$ and the equivalence definitions are defined as follows:
\begin{enumerate} 
\item $\{\Psi_{\pm1,\eps,s},\dots,\Psi_{\pm k,\eps,s} \}_{\eps\in\Omega_s\cup\{0\}} \equiv\{\widetilde{\Psi}_{\pm1,\eps,s},\dots,\widetilde{\Psi}_{\pm k,\eps,s}\}_{\eps\in\Omega_s\cup\{0\}}$ if there exists $B_{\eps,s}$ analytic in $\eps\in \Omega_s$ with continuous limit at $\eps=0$ such that 
$$ \widetilde{\Psi}_{\ell,\eps,s}= T_{-B_{\eps,s}}\circ \Psi_{\ell,\eps,s} \circ T_{B_{\eps,s}}.$$
\item Let $r_{2\ell}\in R_k$, $\ell=0, \dots, k-1$, act on $\eps$ by  
$$r_{2\ell}(\eps_{k-1}, \dots, \eps_1,\eps_0)= \left(\eps_{k-1}e^{-i\frac{2\pi\ell(k-2)}{k}}, \dots, \eps_1,\eps_0e^{i\frac{2\pi\ell}{k}}\right).$$
Let $\Omega_{r_{2\ell}(s)}:= r_{2\ell}(\Omega_s)$. Then
\begin{align*}\begin{split} &\left(k,\eps,b_\eps, \left(\{\Psi_{1,\eps,s}, \Psi_{-1,\eps,s}, \dots,\Psi_{k,\eps,s},\Psi_{-k,\eps,s}\}_{\eps\in\Omega_s\cup\{0\}}\right)_s\right)\\
&\qquad\cong
\left(k,r_{2\ell}(\eps),b_{r_{2\ell}(\eps)}, \left(\{\Psi_{r_{2\ell}(1),r_{2\ell}(\eps),r_{2\ell}(s)}, \Psi_{\xi_{2\ell}(r_{2\ell}(1)),r_{2\ell}(\eps),r_{2\ell}(s)},\dots,\right.\right.\\
&\qquad\qquad\qquad\qquad\qquad\left.\left.\Psi_{r_{2\ell}(k),r_{2\ell}(\eps),r_{2\ell}(s)}, \Psi_{\xi_{2\ell}(r_{2\ell}(k)),r_{2\ell}(\eps),r_{2\ell}(s)}\}_{\eps\in\Omega_s\cup\{0\}}\right)_s\right).\end{split}\end{align*}
\end{enumerate}
\end{definition}

\begin{theorem}\label{thm:sqr} Let $g_\eps$ be a prepared generic unfolding of a holomorphic parabolic germ of type \eqref{hol_parabolic_unf}. Then $g_\eps$ has an antiholomorphic square root $f_\eps$  (i.e. satisfying $f_{\epsbar}\circ f_\eps= g_\eps$), with $f_0$ of the form \eqref{square-root_theta}, if and only if a representative of the modulus \eqref{modulus} satisfies 
\begin{equation} \STt \circ \Psi_{\ell,\eps,s}=\Psi_{s_m(\ell),\epsbar,\ov{s}}\circ \STt.\label{cond:square_root}\end{equation}
Moreover, this antiholomorphic square root is unique unless the modulus is trivial, i.e. $g_\eps$ is conjugate to $v_\eps^1$. In the latter case, there exist an infinite number of square roots which are the conjugates of $r_m\circ \sigma\circ v^{\frac12+iy(\eps)}\circ r_m^{-1}$, with $y(\eps)$ analytic and $y(\epsbar)=\ov{y(\eps)}$. 
\end{theorem}
\begin{proof} When the modulus is trivial we can suppose that $g_\eps= v_\eps^1$. Moreover $(r_m)^*(v_\eps)=(-1)^mv_\eps$. Hence for $m$ odd, $r_m^{-1}\circ g_\eps\circ r_m=v_\eps^{-1}$ and in this case, we consider square roots of $g_\eps^{-1}$. It therefore suffices to consider antiholomorphic square roots tangent to the identity. In the time coordinate (the $Z_j$-coordinate), $v_\eps^1$ is given by $T_1$, and in the coordinate $w= {\rm Exp}(-2\pi i Z)$, it is given by the identity on $\CP^1$. For real $\eps$, square roots in the $w$-coordinate must satisfy  $\kappa_\eps\circ \kappa_\eps= {\rm id}$. Moreover, $\kappa$ exchanges $0$ and $\infty$.
Hence, $\kappa= \delta\circ L$, for $\delta(w) = \frac1{\ov{w}}$ and $L$ some linear transformation. Then, square roots in the $Z_j$-coordinates are of the form $\Sigma\circ T_{a(\eps)}$, with $a(\eps)+\ov{a(\eps)}=1$, i.e. $a(\eps)=\frac12+iy(\eps)$ for some real function $y$ depending real-analytically on $\eps$.  Hence, for real $\eps$,  the square roots  are given by $r_m\circ \sigma\circ v^{\frac12+iy(\eps)}\circ r_m^{-1}$, with $y(\eps)$ real-analytic. The result follows by extending holomorphically $y(\eps)$ to the complex domain (thus yielding that the square root depends antiholomorphically on $\eps$).

Let us now suppose that the modulus is not trivial. As a first reduction, let us rather consider $g_{1,\eps}=r_m^{-1}\circ g_\eps\circ r_m$. Then we can limit ourselves to the case $m=0$ and $\xi_0(j)=-j$. However for $m$ odd,   then $g_{1,0}'(z) = z-z^{k+1} + \dots$ and a second reduction is needed.  When $k$ is odd it suffices to conjugate with $z\mapsto -z$. When $k$ is even, the second reduction is to work with $g_{2,\eps}=g_{1,\eps}^{-1}$, which is in prepared form \eqref{hol_parabolic_unf}.
Hence we can limit ourselves to prove the theorem when $m=0$.   

Let $\Omega_s$ be a sectoral domain, and let $\Phi_{j,\eps,s}$, $j=0,\pm1, \dots, \pm k$ (with indices $(\text{mod}\; 2k)$), be corresponding normalized Fatou coordinates for which the transition functions satisfy \eqref{cond:square_root}. 
We define 
\begin{equation}f_{\eps,s}=  Z_{\epsbar,-j}^{-1}\circ \Phi_{-j,\epsbar,\ov{s}}^{-1} \circ \Sigma\circ T_{\frac12} \circ \Phi_{j,\eps,s}\circ Z_{\eps,j}, \qquad \text{on\ } S_{j,\eps,s}.\label{def:f_s}\end{equation}
Note that Fatou coordinates such that \eqref{cond:square_root} is satisfied are determined up to left composition with $T_{a_s(\eps)}$ such that $a_{\ov{s}}(\epsbar)=\ov{a_s(\eps)}$. 
Hence, the definition of $f_{\eps,s}$ is intrinsic and does not depend on the choice of Fatou coordinates. 
Moreover, the function $f_{\eps,s}$ is well defined on $\D_r\setminus\{P_\eps(z)=0\}$. Indeed \eqref{cond:square_root} guarantees that it is well defined when crossing an intersection sector touching the boundary of the disk because of \eqref{cond:square_root}.
Over a gate sector, it follows from the proofs  of Theorem~\ref{class_thm} and \ref{thm:inv_curve} that the translations $\mathcal{T}_{\ell,\eps,s}$  satisfy 
$\mathcal{T}_{\ell,\eps,s}\circ \Sigma = \Sigma \circ \mathcal{T}_{-\ell,\epsbar, \ov{s}}$.

The map $f_{\eps,s}$ is bounded in the neighborhood of $P_\eps(z)=0$ and hence can be extended to that set. 

We now need to show that different $f_{\eps,s}$ glue into a global $f_\eps$ defined for $\eps$ outside the discriminant set in $\eps$-space. 

It is of course possible to choose the Fatou coordinates respecting \eqref{eq:determine Phi} so that $  \lim_{\eps\to 0\atop \eps\in \Omega_s} \Phi_{\ell,\eps,s} =\Phi_{\ell,0}$ be independent of $s$.  

Let now consider $\Omega_s\cap \Omega_{s'}$ and let $h_\eps= f_{\eps,s'}^{-1}\circ f_{\eps,s}$. It commutes with $g_\eps$. Because the modulus is non trivial and in view of Proposition~\ref{prop:symmetries} this means that $h_\eps= g_\eps^{\frac{p}{q}}$  for some $\frac{p}{q}\in \Z$ independent of $\eps$. Moreover, because of the limit property, then $\lim_{\substack{\eps\to 0,\\ \eps\in\Omega_s\cap \Omega_{s'}}} h_\eps= {\rm id}$. 
Hence $\frac{p}{q}=0$ and  $f_{\eps,s} = f_{\eps,s'}$. 

Finally $f_\eps$ is bounded in the neighborhood of the discriminant set in $\eps$-space and can be extended antiholomorphically there. 
\end{proof}

\begin{corollary}\label{cor:square} Let $g_\eps$ be a prepared unfolding of a holomorphic parabolic germ of type \eqref{hol_parabolic_unf}, for which a representative of the modulus satisfies modulus \eqref{cond:square_root}. Then $g_\eps$ has a holomorphic square root is and only $g_\eps$ has an invariant germ of real analytic curve. 
\end{corollary} 
\begin{proof} By Theorem~\ref{thm:sqr}, $g_\eps$ has a square root $f_\eps$, and then the result  follows from Theorem~\ref{thm:inv_curve}.\end{proof}

\subsection{Application to holomorphic quadratic germs}

\begin{theorem} The holomorphic quadratic parabolic germ $g(z) = z+z^2$ has no antiholomorphic square root, nor any of the $g_\eps(z) = z+z^2-\eps$ for small $\eps$. 
\end{theorem}
\begin{proof} By Theorem~\ref{thm:sqr} and Corollary~\ref{cor:square}, since the real axis is invariant, it suffices to prove that $g$ has no holomorphic square root. 
Suppose that $g$ has a local holomorphic square root $g_1$. In the Fatou coordinates, this square root becomes $T_{\frac12}$. Hence the square root can be extended (not necessarily as a univalent map) in all the domains of extensions of the Fatou coordinates of $g$, i.e. up to the Julia set, which is the closure of the set of repelling periodic points. 
Then, for each periodic orbit  $\{z_1, \dots, z_n\}$ of period $n$ of $g$, $\{g_1(z_1), \dots, g_1(z_n)\}$ is also a periodic orbit of $g$ of period $n$. 
Moreover,  $\{\ov{z_1}, \dots ,\ov{z_n}\}$ and $\{\ov{g_1(z_1)}, \dots, \ov{g_1(z_n)}\}$ are also periodic orbits since $\sigma\circ g=g\circ \sigma$. Hence, generically, except for a few symmetric cases, the number of orbits of a given period $n$ should be a multiple of $4$. 
Let us show that this number is never a multiple of $4$ when $n$ is a prime number. Indeed, periodic points of period $n$ are solutions of $g^{\circ n}(z) - z=0$. This equation has $2^n$ solutions including a double root at $z=0$, hence $2^n-2$ periodic points of period $n$. But $2^n-2 \equiv 2 \;({\rm mod}\: 4)$. 

If follows that the transition maps of $g$ are not periodic of period $\frac12$. Since the transition maps of $g_\eps$ depend continuously on $\eps$, they do not satisfy \eqref{cond_square-root}, which is a necessary condition for $g_\eps$ to have a holomorphic square root. \end{proof}

\section{The multicorn families}

It is shown in \cite{HS} that all parabolic points of the multicorn family $f_c(z)=\ov{z}^d +c$ have multiplicity $1$ or $2$. We give a second proof and add that the family is a generic unfolding around these points. 

\begin{proposition} For $d\geq 2$, the multicorn family $f_c(z)=\ov{z}^d +c$ has $d+1$ values of $c$ given by $c_\tau= (d+1)d^{-\frac{d}{d-1}}e^{i\frac{\pi}{d+1}}\tau$, where $\tau^{d+1}=1$, for which the point $z_\tau= d^{-\frac1{d-1}}e^{i\frac{\pi}{d+1}}\tau$ is an antiholomorphic parabolic point of codimension 2. The family is generic around these points when considering the real and imaginary parts of $c$ as parameters. 
The parabolic fixed points occurring for other parameter values of $c$ have codimension $1$ and the $2$-parameter family contains a generic unfolding around these points. 
\end{proposition} 
\begin{proof} We let $d=k+1$ to use the same notations as in the rest of the paper. A parabolic point of codimension greater than 1 is one for which, under the form $f_1(z_1)= \ov{z}_1 +a_2\ov{z}_1^2 + O(\ov{z_1}^3)$, then $a_2\in i\R$. 
We look for a parabolic point $z_0$, i.e. a fixed point satisfying  $f_{c_0}(z_0)=\ov{z}_0^{k+1} +c_0 =z_0$ and $|f_{c_0}'(z_0)|=|(k+1)\ov{z}_0^k|=1$ for some $c_0\in C$. 
Then $z_0= (k+1)^{-\frac1k} e^{i\theta}$ for some $\theta \in [0,2\pi]$, from which $c_0= z_0- \ov{z}_0^{k+1}$. We localize at $z_0$ by the change of variable $Z=z-z_0$. In the new variable the function becomes 
$$F_{c_0}(Z)= e^{-ik\theta} \ov{Z} + \frac{k(k+1)^{\frac2{k+1}}}2e^{-i(k-1)\theta}\ov{Z}^2+
\frac{k(k-1)(k+1)^{\frac3{k+1}}}6e^{-i(k-2)\theta}\ov{Z}^3+O(\ov{Z}^4).$$
We let $Z_1= e^{i\frac{k\theta}2}Z$. This transforms $F_{c_0}$ into 
\begin{align*}\begin{split}F_{1, c_0}(Z_1)&= \ov{Z}_1 + \frac{k(k+1)^{\frac2{k+1}}}2e^{i\frac{k+2}2\theta}\ov{Z}_1^2\\ &\quad+
\frac{k(k-1)(k+1)^{\frac3{k+1}}}6e^{i(k+2)\theta}\ov{Z}_1^3+O(\ov{Z}_1^4).\end{split}\label{multicorn_cod-2}\end{align*}
Then $Z_1=0$ has codimension $1$ if $e^{i\frac{k+2}2\theta}\notin i \R$ (see for instance \cite{GR21}), and at least $2$ if $e^{i\frac{k+2}2\theta}\in i \R$, i.e. $\theta=\frac{\pi}{k+2}+\frac{2m\pi}{k+2}$ for $m\in\Z_{k+2}$. 
In the latter case, $\ov{z}_0^{k+1}$ is opposite to $z_0$ and $c_0=z_0-\ov{z}_0^{k+1}= (k+2)(k+1)^{-\frac{k+1}k}e^{i\frac{\pi}{k+2}}\tau$ for some $\tau$ satisfying $\tau^{k+2}=1$. 

Note that $F_{1,c_0}(Z_1)= \ov{Z}_1+a_2\ov{Z}_1^2+a_3\ov{Z}_1^3+O(\ov{Z}_1^4)$, with $a_2\in i\R$ and $a_3\in\R_{\leq0}$. In order to check that the codimension is exactly 2 we get rid of the coefficient in $\ov{Z}_1^2$  by means of the change of coordinate $Z_1=Z_2+ \frac{a_2}2Z_2^2$. This transforms $F_{1,c_0}$ into $F_{2,c_0}(Z_2)=\ov{Z}_2+(a_3-a_2^2) \ov{Z}_2^3+O\left(\ov{Z}_2^4\right)$. 
Then $$a_3-a_2^2=\frac{k(k+1)^{\frac3{k+1}}}{12}\left(3k(k+1)^{\frac1{k+1}}-2(k-1)\right)>0.$$ 

We now consider the family in the neighborhood of $c_0$, by taking  $c= c_0+\eps$. Then $F_c(Z)= F_{c_0}(Z) +\eps$, and $F_{1,c}(Z_1)= F_{1,c_0}+\eps e^{i\frac{k\theta}2}$. 
Finally the change  $Z_1=Z_2+ \frac{a_2}2Z_2^2$ brings it to $$F_{2,c}(Z_2)= \eta_0+(1+\eta_1) \ov{Z}_2+ O(\eta)\ov{Z}_2^2+\left(a_3-a_2^2+O(\eta)\right)\ov{Z}_2^3+O\left(\ov{Z}_2^4\right),$$ where $\eta_0=e^{i\frac{k\theta}2} \eps +o(\eps)$ and $\eta_1=-a_2\eta_0+o(\eta_0)$. A further scaling $Z_2\mapsto rZ_2$ for some $r\in\R_{>0}$ would change $F_{2,\eps}$ exactly to the form \eqref{gen_unfolding}. The change of parameter $(\Re(\eps),\Im(\eps)) \mapsto (\Re(\eta_0),\Re(\eta_1))$ is invertible, since $a_2\in i\R$, from which the genericity of the family follows. 

In the codimension $1$ case the corresponding change $Z_1=Z_2+ i\frac{\Im (a_2)}2Z_2^2$ brings the family to the form
$$F_{2,c}(Z_2)= \eta_0+(\Re (a_2)+\eta_1) \ov{Z}_2+ O(\eta)\ov{Z}_2^2+\left(a_3-i a_2\Im( a_2)+O(\eta)\right)\ov{Z}_2^3+O\left(\ov{Z}_2^4\right),$$
which is a $2$-parameter unfolding containing a generic unfolding. \end{proof}

\section*{Acknowledgements}

The author is greatful to Arnaud Ch\'eritat, Jonathan Godin and Martin Klime\v{s} for stimulating discussions.


\begin{thebibliography}{MRR04}

\bibitem[CR14]{CR} C.~Christopher and C.~Rousseau,
\emph{The moduli space of germs of generic families of analytic
  diffeomorphisms unfolding a parabolic fixed point}, 
 International Mathematics Research Notices, \textbf{2014} (2014), no. 9, 2494--2558.

 \bibitem[DES05]{DES05} A.~Douady,  J.F.~Estrada, P.~Sentenac, \emph{Champs de vecteurs polynomiaux sur $\mathbb{C}$}, unpublished manuscript (2005).

\bibitem[DH84]{DH1} A.~Douady, J, Hubbard, \emph{\'Etudes dynamique des polyn\^omes complexes, Partie I}, volume~84-2 of {\em Publications Math\'ematiques d'Orsay}, Universit\'{e} de Paris-Sud, D\'{e}partement de Math\'{e}matiques,
  Orsay, 1984.

\bibitem[DH85]{DH2} A.~Douady, J, Hubbard, \emph{\'Etudes dynamique des polyn\^omes complexes, Partie II},  volume~85-4 of {\em Publications Math\'ematiques d'Orsay}, Universit\'{e} de Paris-Sud, D\'{e}partement de Math\'{e}matiques,
  Orsay, 1985.

\bibitem[E85]{E}
J.~\'{E}calle, \emph{Les fonctions r\'{e}surgentes, Tome {III}}, volume~85-5 of {\em
  Publications Math\'{e}matiques d'Orsay},
Universit\'{e} de Paris-Sud, D\'{e}partement de Math\'{e}matiques,
  Orsay, 1985.

\bibitem[Gl01]{Gl}  A. A.~Glutsyuk, Confluence of singular points and nonlinear Stokes
phenomenon, \emph{Trans. Moscow Math. Soc.} \textbf{62} (2001),
49--95.

\bibitem[GR21]{GR21} J.~Godin and C. ~Rousseau, \emph{Analytic classification of germs of parabolic antoholomorphic diffeomorphisms of codimension $k$},  Ergodic Theory Dynam. Systems (2021), \url{https://www.doi.org/10.1017/etds.2021.98}. 

\bibitem[GR22]{GR22} J.~Godin and C. ~Rousseau, \emph{Analytic classification of generic unfoldings of antiholomorphic parabolic fixed points of codimension 1}, to appear in Moscow Mathematical Journal, preprint 2021, arXiv:2105.10348.

\bibitem[HS14]{HS}
J.H. Hubbard and D.~Schleicher,
\emph{Multicorns are not path connected}, in {\em Frontiers in complex dynamics}, volume~51 of {\em Princeton
  Math. Ser.}, pages 73--102. Princeton Univ. Press, Princeton, NJ, 2014.
  
 \bibitem[I93]{I}
Y.S. Ilyashenko,
\emph{Nonlinear Stokes Phenomena}, in {\em Nonlinear Stokes phenomena},  Adv. Soviet Math., 14, pages 1--55, Amer. Math. Soc., Providence, RI, 1993. 

 \bibitem[IY08]{IY}
Y.S. Ilyashenko and S.~Yakovenko,
\emph{Lectures on Analytic Differential Equations},
 Graduate studies in mathematics, American Mathematical Society, 2008.

\bibitem[IM16]{IM}
H.~Inou and S.~Mukherjee,
\emph{Non-landing parameter rays of the multicorns}, Invent. Math., \textbf{204} (2016), no. 3, 869--893.

\bibitem[KR20]{KR20} M.~Klime\v{s} and C.~Rousseau, \emph{On the universal unfolding of vector fields in one variable: a proof of Kostov's theorem}, Qual. Theory Dyn. Syst. 19 (2020), no. 3.

\bibitem[L89]{L} P.~Lavaurs, \emph{Syst\`emes dynamiques holomorphes: explosion de points
p\'eriodiques paraboliques} Thesis, Universit\'e de
Paris-Sud, 1989.

\bibitem[MRR94]{MRR}
P.~Marde{\v s}i{\'c}, R.~Roussarie, and C.~Rousseau,
\emph{Modulus of analytic classification for unfoldings of generic
  parabolic diffeomorphisms},  Moscow Mathematical Journal, \text{bf} (2004), no. 2, 455--502.

\bibitem[Ma87]{Ma}  J.~Martinet,  Remarques sur la bifurcation n\oe ud-col
dans le domaine complexe, \emph{Ast\'erisque} \textbf{150-151}
(1987), 131--149.

\bibitem[Mi92]{M}
J.~Milnor, \emph{Remarks on Iterated Cubic Maps}, Exper. Math. \textbf{1} (1992), no. 1, 5--25.

\bibitem[MNS17]{MNS}
S.~Mukherjee, S.~Nakane, and D.~Schleicher,
\emph{ On multicorns and unicorns {II}: bifurcations in spaces of
  antiholomorphic polynomials},  Ergodic Theory Dynam. Systems, \textbf{37} (2017), no. 3, 859--899.

\bibitem[NS03]{NS}
S.~Nakane and D.~Schleicher,
\emph{On multicorns and unicorns. {I}. {A}ntiholomorphic dynamics,
  hyperbolic components and real cubic polynomials}, Internat. J. Bifur. Chaos Appl. Sci. Engrg., \textbf{13} (2003), no. 10, 2825--2844.
  
 \bibitem[O99]{O} R.~Oudkerk,  The parabolic
implosion for $f_0(z)=z+z^{\nu+1}+O(z^{\nu+2})$, thesis, University
of Warwick (1999).

\bibitem[Ri08]{Ri}  J.~Rib\'on, Modulus of analytic classification for unfoldings of resonant
diffeomorphisms,  \emph{Moscow Mathematical Journal},  \textbf{8}
(2008), 319--395, 400.

\bibitem[Ro15]{R15} C.~Rousseau, \emph{Analytic moduli for unfoldings of germs of generic analytic diffeomorphisms with a codimension $k$ parabolic point},  Ergodic Theory Dynam. Systems \textbf{35} (2015), no. 1, 274--292. 

\bibitem[S00]{S} M.~Shishikura, \emph{Bifurcation of parabolic fixed
points}, in \emph{The Mandelbrot set, theme and variations}, London
Math. Soc. Lecture Note Ser., vol 274,
2000, pages 325--363.



\bibitem[V81]{V}
S.~M. Voronin, \emph{Analytic classification of germs of conformal mappings $(\C,O) \to
  (\C,O)$ with identity linear part},
Funct. Anal. Appl., \textbf{15} (1981), 1--13.


\end{thebibliography}
\end{document}